\documentclass[12pt, reqno]{amsart}
\usepackage{amsmath, amstext, amsbsy, amssymb, amscd}
\usepackage{amsmath}
\usepackage{amsxtra}
\usepackage{amscd}
\usepackage{amsthm}
\usepackage{amsfonts}
\usepackage{amssymb}
\usepackage{eucal}
\usepackage{color}

\input prepictex
\input pictex
\input postpictex

\newtheorem{theorem}{Theorem}[section]
\newtheorem{procedure}[theorem]{Procedure}
\newtheorem{lemma}[theorem]{Lemma}
\newtheorem{proposition}[theorem]{Proposition}
\newtheorem{corollary}[theorem]{Corollary}
\theoremstyle{definition}
\newtheorem{definition}[theorem]{Definition}

\theoremstyle{remark}
\newtheorem{remark}[theorem]{Remark}
\newtheorem{conjecture}[theorem]{Conjecture}
\definecolor{A}{rgb}{.75,1,.75}

\numberwithin{equation}{section}

\newcommand{\B}{\mathfrak B}
\newcommand{\C}{ \mathbb C }

\newcommand{\Fi}{\mathcal O_{\m|\infty}^{++}}

\newcommand{\FPn}{\mathcal O_{\m|n}^{+}}
\newcommand{\FnDelta}{\mathcal O_{\m|n}^{+,\Delta}}
\newcommand{\FPPn}{\mathcal O_{\m|n}^{++}}

\newcommand{\Oi}{\mathcal O_{\m+\infty}^{++}}
\newcommand{\OnDelta}{\mathcal O_{\m+n}^{+,\Delta}}
\newcommand{\OPn}{\mathcal O_{\m+n}^{+}}
\newcommand{\OPPn}{\mathcal O_{\m+n}^{++}}

\newcommand{\Xmi}{X_{\m|\infty}}

\newcommand{\Zmn}{ \Z_+^{\m|n}}
\newcommand{\Zmpn}{ \Z_+^{\m+n}}

\newcommand{\glmpn}{{\mathfrak g\mathfrak l}(m+n)}
\newcommand{\g}{\gamma}
\newcommand{\gl}{{\mathfrak g\mathfrak l}}
\newcommand{\glmn}{{\mathfrak g\mathfrak l}(m|n)}
\newcommand{\glsuper}{{\mathfrak g\mathfrak l}(m|\infty)}
\newcommand{\hf}{\frac12}
\newcommand{\la}{\lambda}
\newcommand{\La}{\Lambda}
\newcommand{\m}{{\bf m}}
\newcommand{\n}{{\bf n}}
\newcommand{\N}{\mathbb N}
\newcommand{\pr}{\text{pr}}

\newcommand{\wt}{ \text{wt} }
\newcommand{\wgv}{\Lambda^{\infty} \mathbb V}
\newcommand{\wgw}{\Lambda^{\infty} \mathbb V^*}
\newcommand{\Z}{ \mathbb Z }

{\vskip-\lastskip\medskip
  \noindent
  {\em #1.}\enspace
  }%
{\qed\par\medskip
  }

\begin{document}
\title[Brundan-Kazhdan-Lusztig and super duality conjectures]
{Brundan-Kazhdan-Lusztig and super duality conjectures}

\author[Shun-Jen Cheng]{Shun-Jen Cheng}
\address{Institute of Mathematics, Academia Sinica, Taipei, Taiwan
11529}\email{chengsj@math.sinica.edu.tw}

\author[Weiqiang Wang]{Weiqiang Wang}
\address{Department of Mathematics, University of Virginia,
Charlottesville, VA 22904} \email{ww9c@virginia.edu}
%\thanks{Partially supported by an NSF grant}

\subjclass[2000]{Primary 17B10, Secondary 17B37, 20C08}

\begin{abstract}
We formulate a general super duality conjecture on connections
between parabolic categories $\mathcal O$ of modules over Lie
superalgebras and Lie algebras of type $A$, based on a Fock space
formalism of their Kazhdan-Lusztig theories which was initiated by
Brundan. We show that the Brundan-Kazhdan-Lusztig (BKL)
polynomials for $\glmn$ in our parabolic setup can be identified
with the usual parabolic Kazhdan-Lusztig polynomials. We establish
some special cases of the BKL conjecture on the parabolic category
$\mathcal O$ of $\glmn$-modules and additional results which
support the BKL conjecture and super duality conjecture.
\end{abstract}

\maketitle

\setcounter{tocdepth}{1} \tableofcontents

\date{}

\tableofcontents

\section{Introduction}

\subsection{The earlier work}

In 2003 Brundan \cite{Br} obtained a purely algebraic and conceptual
solution to the problem of finding finite-dimensional irreducible
characters of the complex Lie superalgebra $\gl(m|n)$. Earlier
Serganova \cite{Se} found an algorithm for computing these
irreducible characters using a mixture of algebraic and geometric
technique. This problem can be traced back three decades earlier to
Kac \cite{K1, K2}, where initial progress was made. In the meantime,
there has been a tremendous amount of work towards it with various
partial results (see \cite{Se, Br, CWZ} for more references).

In Brundan's approach, the Hecke algebra modules and their
bar-invariant basis in the standard Kazhdan-Lusztig (KL) theory
\cite{KL, KL2, Deo} are replaced by the module $\Lambda^m \mathbb V
\bigotimes \Lambda^n \mathbb V^*$ over the quantum group $\mathcal
U_q\mathfrak{sl}_\infty$ and its Lusztig-Kashiwara canonical
basis/global basis, where $\mathbb V$ denotes the natural $\mathcal
U_q\mathfrak{sl}_\infty$-module. The {\em Fock space} $\Lambda^m
\mathbb V \bigotimes \Lambda^n \mathbb V^*$ at $q=1$ should be
regarded as the Grothendieck group of the category $\mathcal
O^+_{m|n}$ of finite-dimensional $\gl(m|n)$-modules. Such a Fock
space approach has been further applied successfully to study the
finite-dimensional irreducible and tilting characters of other Lie
superalgebras \cite{Br3, CWZ2}. In \cite{Br}, for the first time, a
Kazhdan-Lusztig conjecture for the full category $\mathcal O$ of
$\gl(m|n)$-modules is formulated using the canonical basis theory of
the module $\mathbb V^{\otimes m} \bigotimes (\mathbb V^*)^{\otimes
n}$.

Subsequently in a joint work \cite{CWZ} of the authors with Zhang,
a connection between $\mathcal O^+_{m|n}$ and the parabolic
category $\mathcal O^+_{m+n}$ of $\gl(m+n)$-modules, associated
with the maximal parabolic subalgebra $\mathfrak p_{m,n}$, was
formulated. Roughly speaking, by developing further the Fock space
formalism we showed that for a fixed $m$ the inverse limits
$\lim\limits_{{\longleftarrow}} \mathcal O^+_{m|n}$ and
$\lim\limits_{{\longleftarrow}} \mathcal O^+_{m+n}$, with respect
to $n$, afford isomorphic Kazhdan-Lusztig theories, and moreover,
we conjectured an equivalence of the two categories.

\subsection{The conjectures}

Fix an $s$-tuple of positive integers $\m =(m_1,\ldots,m_s)$ with
$\sum_a m_a =m$. In the present paper we formulate a parabolic
version of Brundan's conjecture for a category $\mathcal O^+_{\m|n}$
of $\gl(m|n)$-modules with respect to a fairly general parabolic
subalgebra $\mathfrak p_{\m,n}$, and a super duality conjecture on
the equivalence of categories of $\lim\limits_{{\longleftarrow}}
\mathcal O^+_{\m|n}$ and $\lim\limits_{{\longleftarrow}} \mathcal
O^+_{\m+n}$, where $\mathcal O^+_{\m+n}$ stands for an analogous
parabolic category of $\gl(m+n)$-modules.

According to this version of Brundan-Kazhdan-Lusztig (BKL)
conjecture, the parabolic Verma, tilting, and irreducible modules in
$\mathcal O^+_{\m|n}$ correspond respectively to the monomial,
canonical, and dual canonical basis elements in the Fock space
$\mathcal E^{\m|n} := \bigotimes_a \Lambda^{m_a} \mathbb V
\bigotimes \Lambda^n \mathbb V^*$ (or rather in a suitable
topological completion). On the other hand, one has an increasingly
better known reformulation of the Kazhdan-Lusztig conjecture
(theorem of Beilinson-Bernstein \cite{BB} and Brylinski-Kashiwara
\cite{BK}; see Soergel \cite{So2} for tilting module characters)
that the tilting and irreducible modules in $\mathcal O^+_{\m+n}$
correspond to the canonical and dual canonical basis elements in the
Fock space $\mathcal E^{\m+n} :=\bigotimes_a \Lambda^{m_a} \mathbb V
\bigotimes \Lambda^n \mathbb V$. Alternatively, the KL conjecture
can also be viewed as a special case of the parabolic BKL conjecture
with $n=0$.

{
 Even though the formulation of the above conjectures in such
a parabolic generality seems inevitable or unsurprising to some
experts after the works \cite{Br} and \cite{CWZ}, we hope that the
general reader may still find it worthwhile and helpful, as it
clarifies the scope and the limitation of these new developments.
Sometimes a more general conjecture has a better chance for
(partial) verification as they involve simpler combinatorics
(compare the treatment of parabolic KL polynomials by Deodhar
\cite{Deo} and its impact on the related development of parabolic KL
conjectures).
 }

\subsection{The main results}

We establish various compatibility results on the bar involution,
canonical and dual canonical bases of the Fock spaces $\mathcal
E^{\m|n}$ and $\mathcal E^{\m+n}$, when $n$ varies. In particular
there is a canonical isomorphism of these spaces at the limit $n \to
\infty$. This allows us to identify Brundan's KL polynomials with
the classical type $A$ parabolic KL polynomials. We show that the
canonical basis elements in $\mathcal E^{\m|n}$, and then in
$\mathcal E^{\m+n}$, stabilize in a suitable sense for $n\gg 0$. We
further establish in Theorem \ref{th:BrConj} a positivity result on
the expansion of the divided powers of Chevalley generators acting
on (dual) canonical basis elements, confirming a parabolic version
of \cite[Conjecture~2.28]{Br}. As a corollary it follows that every
canonical basis element in the Fock space $\mathcal E^{\m|n}$ is a
{\em finite} sum of monomials.

In an approach different from \cite{CWZ}, we establish in
Section~\ref{sec:superRep} properties of tilting modules in
$\mathcal O^+_{\m|n}$ for varying $n$ without assuming either the
validity of the BKL conjecture or using explicit formulas of
canonical basis. We introduce truncation functors that interpolate
the categories $\mathcal O^+_{\m|n}$ and $\mathcal O^+_{\m+n}$ for
varying $n$, and establish various compatibility results. In
particular, we prove a stability result for the tilting modules
$U_n(\la)$ in $\mathcal O^+_{\m|n}$ for a given weight $\la$, i.e.,
the $U_n(\la)$ have the same finite Verma flag structures for every
$n\gg 0$ (where it is understood that a tail of zeros is added to
$\la$ for larger $n$). The connections between canonical bases in
various Fock spaces $\mathcal E^{\m|n}$ and $\mathcal E^{\m+n}$
further allow us to establish the same stability result for tilting
modules in $\mathcal O^+_{\m+n}$. (We are not aware of any other
proofs even though such a statement appears to be classical).

The parabolic BKL conjecture for $\mathcal O^+_{\m|n}$ would follow
from the properties of the truncation maps and functors established
in this paper, under the assumption of the validity of the super
duality conjecture. Also, it would follow from the validity of
Brundan's conjecture on the full category $\mathcal O$. { However,
the parabolic formulation of this paper can still be useful, since
most of our results stated above either do not make sense or cannot
be proved for now in the setup of the full category $\mathcal O$ or
its associated Fock space.}

Note that the known proofs of the classical Kazhdan-Lusztig
conjectures ultimately rely on geometric machinery. For lack of such
geometric tools, the BKL conjecture, or the super duality conjecture
in general, presently appears to lie beyond our reach. We obtain
some partial verification of the BKL conjecture under some
``regularity" condition on the weights. In the special case when $\m
=(1,1)$ and $n$ is arbitrary, we establish the parabolic BKL
conjecture and a weak version of the super duality conjecture, where
among others the method of the $\mathfrak{sl}_2$-categorification of
Chuang-Rouquier \cite{CR} is used. We also establish the parabolic
BKL conjecture in another special case when $\m =(m,1)$ and $n=1$.
In both cases, we find explicit formulas for the canonical basis and
thus the weights of Verma flags of the tilting modules. (Our
approach can be adapted to give a purely algebraic proof of the
usual type $A$ Kazhdan-Lusztig conjecture in the corresponding
parabolic and low rank cases).

\subsection{The organization}

The layout of this paper is as follows.

\begin{itemize}
\item In Section~\ref{sec:basic}, we define the canonical and dual
canonical bases for the Fock spaces $\mathcal E^{\m|n}$, and
investigate their relationship for varying $n$ under the truncation
maps.

\item In Section~\ref{sec:superRep}, we formulate the parabolic
BKL conjecture on $\mathcal O^+_{m|n}$ and establish various
results on tilting modules.

\item In Section~\ref{sec:FockredKL}, we reformulate the classical
parabolic Kazhdan-Lusztig conjecture of type $A$ by means of the
Fock space $\mathcal E^{\m+n}$ and also present our general super
duality conjecture. We obtain a key isomorphism result on Fock
spaces which underlies the super duality conjecture.

\item In Section~\ref{sec:categorification}, we adapt the powerful
machinery of the $\mathfrak{sl}_2$-categorification of
Chuang-Rouquier to the category $\mathcal O^+_{\m|n}$. Some formal
consequences of the $\mathfrak{sl}_2$-categorification are used in
the subsequent sections.

\item In Section~\ref{sec:partial}, as a preparation for the next
sections, we establish several technical results regarding the
tilting modules in the category $\mathcal O^+_{\m|n}$. We also
give an explicit description of the tilting modules when the
weights satisfy a {\em regular} condition, which partially
verifies the parabolic BKL conjecture.

\item In Section~\ref{canonical:11n}, we establish the parabolic
BKL conjecture and a weak version of the super duality conjecture
when $\m=(1,1)$. In Section~\ref{sec:categorym11}, we establish
the parabolic BKL conjecture for $\mathcal O^+_{m,1|1}$.

\item In Section~\ref{sec:gl21}, we focus on the category
$\mathcal O_{2|1}^+$ of $\gl(2|1)$-modules. We work out explicitly
the Verma flag structures for the tilting and projective modules, as
well as the composition series of Verma modules. We further classify
the projective tilting modules.
\end{itemize}

We often omit the details of proofs when they are very similar or
even identical to those for the special case (i.e. $\m=m$) treated
in \cite{Br} and \cite{CWZ} to keep the paper within a reasonable
size. The reader is recommended to have copies of these two papers
at hand when reading the present paper.

{\bf Acknowledgments.} S-J.C.~is partially supported by an NSC grant
of the R.O.C.~and an Academia Sinica Investigator grant. He also
thanks the Department of Mathematics, University of Virginia, for
hospitality and support. W.W.~is partially supported by NSF. We are
grateful to Jon Brundan for his influential ideas and several
stimulating discussions. We also thank R.B. Zhang for his
participation at an early stage of this work and in \cite{CWZ}.
Notation: $\N =\{0,1,2,\cdots\}$.

\section{Basics of $q$-multilinear algebras}
\label{sec:basic}

In this section we set up various notations, compatible with
\cite{CWZ} which is our special case when $\m =m$. We refer to
\cite[Section~2]{CWZ} for more detail (also see \cite{Br}).
\subsection{The quantum group}

The quantum group $U_q{\mathfrak g\mathfrak l}_\infty$ is the
$\mathbb Q(q)$-algebra generated by $E_a, F_a, K^{\pm 1}_a, a \in
\Z$, subject to the relations

\begin{eqnarray*}
 K_a K_a^{-1} =K_a^{-1} K_a =1, &&
 K_a K_b = K_b K_a, \\
 K_a E_b K_a^{-1} = q^{\delta_{a,b} -\delta_{a,b+1}} E_b, &&
 K_a F_b K_a^{-1} = q^{\delta_{a,b+1}-\delta_{a,b}}
 F_b, \\
 E_a F_b -F_b E_a &=& \delta_{a,b} (K_{a,a+1}
 -K_{a+1,a})/(q-q^{-1}), \\
 E_a E_b = E_b E_a,   &&
 F_a F_b = F_b F_a,  \qquad\qquad \text{if } |a-b|>1,\\
 E_a^2 E_b +E_b E_a^2 &=& (q+q^{-1}) E_a E_b E_a,  \qquad\qquad \text{if } |a-b|=1, \\
 F_a^2 F_b +F_b F_a^2 &=& (q+q^{-1}) F_a F_b F_a,  \qquad\qquad \text{if } |a-b|=1.
\end{eqnarray*}
Here and below $K_{a,b} :=K_aK_{b}^{-1}$ for $a \neq b \in \Z$.
%Denote $F_{a}^{(r)}:= F_{a}^{r}/[r]!$ and
%$E_{a}^{(r)}:= E_{a}^{r}/[r]!$, where $[r]=(q^r
%-q^{-r})/(q-q^{-1})$ and $[r]! =[r][r-1]\cdots [1]$.
Define the bar involution on $U_q{\mathfrak g\mathfrak l}_\infty$
to be the anti-linear automorphism $ ^-:  E_a\mapsto E_a, \quad
F_a\mapsto F_a, \quad K_a\mapsto K_a^{-1}.$
Here by {\em anti-linear} we mean with respect to the automorphism
of $\mathbb Q(q)$ given by $q \mapsto q^{-1}$.

Let $\mathbb V$ be the natural $U_q{\mathfrak g\mathfrak
l}_\infty$-module with basis $\{v_a\}_{a\in\Z}$ and $\mathbb W
:=\mathbb V^*$ the dual module with basis $\{w_a\}_{a\in\Z}$ such
that $w_a (v_b) = (-q)^{-a} \delta_{a,b}$.  We have
%The actions of
%$U_q({\mathfrak g\mathfrak l}(\infty))$ on $\mathbb V$ and
%$\mathbb W$ are given by the following formulas:
\begin{align*}
K_av_b=q^{\delta_{ab}}v_b,\quad E_av_b=\delta_{a+1,b}v_a,\quad
F_av_b=\delta_{a,b}v_{a+1},\\
K_aw_b=q^{-\delta_{ab}}w_b,\quad E_aw_b=\delta_{a,b}w_{a+1},\quad
F_aw_b=\delta_{a+1,b}w_{a}.
\end{align*}
As in \cite{Br, CWZ} we shall use the comultiplication $\Delta$ on
$U_q{\mathfrak g\mathfrak l}_\infty$ defined by:
\begin{eqnarray*}
 \Delta (E_a) &=& 1 \otimes E_a + E_a \otimes K_{a+1, a}, \\
 \Delta (F_a) &=& F_a \otimes 1 +  K_{a, a+1} \otimes F_a, \quad
 \Delta (K_a) = K_a \otimes K_{ a}.
\end{eqnarray*}
We let $\mathcal U = U_q{\mathfrak s\mathfrak l}_\infty$ denote
the subalgebra with generators $E_a$, $F_a$, $K_{a,a+1}, a\in \Z$.

For $k \ge 0$, set $[k]=\frac{q^k-q^{-k}}{q-q^{-1}}$ and $[k]!
=[k][k-1]\cdots [1]$, and introduce the divided power $E_a^{(k)}
=E_a^{k}/[k]!, F_a^{(k)} =F_a^{k}/[k]!$. One has the following
comultiplication formula
\begin{eqnarray*}
\Delta(F_a^{(k)}) = \sum_{i=0}^k  q^{i(k-i)} K_{a, a+1}^i
F_a^{(k-i)} \otimes F_a^{(i)}.
\end{eqnarray*}

\subsection{The Fock space $\mathcal E^{\m+n}$}

For $m \in \N, n\in\N \cup \infty$, we let
$$I(m|n):=\{-m,-m+1,\ldots,-1\}\cup\{1,2,\ldots,n\}.
$$
Given an $s$-tuple of positive integers
$$\m =(m_1,\ldots,m_s), \qquad \text{where }  m_1 +\cdots + m_s=m,$$
we denote by $S_{m+n}$ the symmetric group of (finite) permutations
on $I(m|n)$, by $S_{\m|n}$ its Young subgroup $S_{m_1}\times\cdots
S_{m_s} \times S_n$, and by $w_0$ the longest element in $S_{\m|n}$
for $n$ finite. Denote by $\tau_{ij}$ the transposition
interchanging $i$ and $j$.

For $n \in \N \cup \infty$, we let $\Z^{m+n}$ or $\Z^{m|n}$ be the
set of integer-valued functions on $I(m|n)$.
 Set (for a finite $n$)
\begin{align*}
\Z_+^{\m+n} &:= \{f \in \Z^{m+n} \mid f(-m) > \cdots> f(-m+m_1-1),
\\
& \qquad\quad f(-m+m_1)>  \cdots > f(-m+m_1+m_2-1),  \\
& \qquad\quad \ldots, f(-m_s) >  \cdots > f(-1), f(1) > \cdots > f(n)\},   \\
\Z^{\m+n}_{++} &:= \{f \in \Z_+^{\m+n} \mid  f(n) \ge 1-n \}.\\
\Z_+^{\m+\infty} &:= \{f \in \Z^{m+\infty} \mid f(-m) > \cdots>
f(-m+m_1-1),
\\
& \qquad\quad f(-m+m_1)>  \cdots > f(-m+m_1+m_2-1), \ldots, \\
& \qquad\quad f(-m_s) >  \cdots > f(-1), f(1)  > f(2) > \cdots;
f(i) =1-i \text{ for } i\gg0 \}.
\end{align*}
Occasionally, we shall denote $\Z_{++}^{\m+\infty} \equiv
\Z_+^{\m+\infty}$.

For $n \in \N$, one can define a right action of the Hecke algebra
$\mathcal H_n$ of type $A$ on the tensor space $\mathbb V^{\otimes
n}$ which commutes with the action via the $(n-1)st$-iterated
comultiplication $\Delta^{n-1}$ of $\mathcal U$ following Jimbo
\cite{Jim}. One can define the space $\Lambda^n \mathbb V$ of
finite $q$-wedges as a quotient space of $\mathbb V^{\otimes n}$
via the skew $q$-symmetrizer from $\mathcal H_n$ and then the
space $\Lambda^\infty \mathbb V$ of infinite-wedges by taking the
limit $n \rightarrow \infty$ appropriately as done in \cite{KMS}.
These spaces are naturally $\mathcal U$-modules. The {\em
$q$-wedge} $v_{a_1} \wedge \cdots \wedge v_{a_n}$ is an element of
$\Lambda^n \mathbb V$, which is the image  of $v_{a_1} \otimes
\cdots \otimes v_{a_n}$ under the canonical map when $\Lambda^n
\mathbb V$ is regarded as a quotient of $\mathbb V^{\otimes n}$.
The elements $v_{a_1} \wedge \cdots \wedge v_{a_n}$, for $a_1 >
\cdots > a_n$ and $a_i\in \Z$, form a basis for $\Lambda^n \mathbb
V$. Similarly, the $\mathcal U$-module $\wgv$ has a basis given by
the infinite $q$-wedges
$v_{m_1} \wedge v_{m_2} \wedge v_{m_3} \wedge \cdots,$
where  $m_1 > m_2  >m_3> \cdots,$ and $m_i = 1-i$ for $i \gg0$
(our $\wgv$ is $F_{(0)}$ in  \cite{KMS}). Alternatively, $\wgv$
has a basis
$$|\la\rangle := v_{\la_1} \wedge v_{\la_2 -1} \wedge v_{\la_3 -2}\wedge \cdots,$$
where $\la =(\la_1, \la_2, \ldots)$ runs over the set of all
partitions.

For $n \in \N \cup \infty$, the space
$$\mathcal E^{\m+n} := \bigotimes_{a=1}^s\Lambda^{m_a} \mathbb V \bigotimes \Lambda^n \mathbb V,$$
is acted upon by $\mathcal U$ via the $s$-th iterated
comultiplication $\Delta^s$.
% and $\Lambda^\m \mathbb V := \bigotimes_{a=1}^s\Lambda^{m_a} \mathbb V.$
It has the {\em monomial basis}
$$\mathcal K_f :=
 v_{f[-m,-m+m_1)} \otimes
 v_{f[-m+m_1,-m+m_1+m_2)} \otimes \cdots  \otimes v_{f[-m_s,-1]} \otimes
 v_{f[1,n]},
$$
where $f$ runs over the set $\Z_+^{\m+n}$ and we have denoted by,
for given $a\leq b$,
$$v_{f[a,b]} \equiv v_{f[a,b+1)} :=v_{f(a)} \wedge v_{f(a+1)}
\wedge \cdots \wedge v_{f(b)}.$$

The {\em Bruhat ordering} $\,\leq$ on $\Z^{m+n}$, which comes from
the Bruhat ordering on $S_{m+n}$, is the transitive closure of the
relation $f< f \cdot \tau_{ij}$, if $f(i) < f(j)$, for $i,j \in
I(m|n)$ with $i<j$. This induces the Bruhat ordering $\leq$ on
$\Z_+^{\m+n}$.

Let $P$ be the free abelian group with basis $\{\epsilon_a\vert
a\in\Z\}$ equipped with a bilinear form $(\cdot | \cdot)$, for which
the $\epsilon_a$'s are orthonormal. For later use, we define the
$\epsilon$-weights on $\Z^{m+n}$:
\begin{equation}\label{nonsuperweight}
\wt^{\epsilon} (f):=\sum_{i\in I(m|n)}\epsilon_{f(i)}, \qquad
\text{for } f \in \Z^{m+n}.
\end{equation}

\subsection{The Fock space $\mathcal E^{\m|n}$}
\label{subsec:superFock}

Set (for a finite $n$)%
\begin{align*}
\Z_+^{\m|n} &:= \{f \in \Z^{m|n} \mid f(-m) > \cdots> f(-m+m_1-1),
\\
& \qquad\quad f(-m+m_1)>  \cdots > f(-m+m_1+m_2-1),  \\
& \qquad\quad \ldots, f(-m_s) >  \cdots > f(-1), f(1)
< \cdots < f(n)\},   \\
\Z^{\m|n}_{++} &:= \{f \in \Z_+^{\m|n} \mid  f(n) \leq n \}.\\
\Z_+^{\m|\infty} &:= \{f \in \Z^{m+\infty} \mid f(-m) > \cdots>
f(-m+m_1-1),
\\
& \qquad\quad f(-m+m_1)>  \cdots > f(-m+m_1+m_2-1), \ldots, \\
& \qquad\quad f(-m_s) >  \cdots > f(-1),
 f(1)  < f(2) < \cdots; f(i) =i \text{ for } i\gg0 \}.
\end{align*}
(Occasionally, we also denote $\Z_{++}^{\m|\infty} \equiv
\Z_+^{\m|\infty}$.)

Recall that $\mathbb W =\mathbb V^*$ is the $\mathcal U$-module dual
to $V$ with basis $\{w_a\}_{a \in \Z}$. The space
$\mathbb W^{\otimes n}$ admits a right action of the Hecke algebra
$\mathcal H_n$ which commutes with the action via $\Delta^{n-1}$
of the quantum group $\mathcal U$. In the same way using the skew
$q$-symmetrizer, the $\mathcal U$-module $\Lambda^n\mathbb W$ has
a basis given by $w_{a_1} \wedge w_{a_2}  \wedge \cdots\wedge
w_{a_n}$ for $a_1<\ldots <a_n$. Similarly, we construct the space
$\Lambda^\infty \mathbb W$ of semi-infinite $q$-wedges
$w_{n_1} \wedge w_{n_2}  \wedge \cdots,$
where $n_i =i$ for $i\gg0$, which carries a $\mathcal U$-module
structure. Writing the conjugate partition of $\la$ as $\la'
=(\la'_1, \la'_2, \ldots)$, we set
$$ |\la'_*\rangle := w_{1 -\la'_1} \wedge w_{2 -\la'_2} \wedge w_{3
-\la'_3} \wedge \cdots.$$ The set $\{|\la'_*\rangle\}$ is a basis
for $\Lambda^\infty \mathbb W$.

For $n \in \N \cup \infty$, we denote
$$\mathcal E^{\m|n} = \bigotimes_{a=1}^s\Lambda^{m_a} \mathbb V  \bigotimes \Lambda^n \mathbb
W,$$ which is acted upon by $\mathcal U$ via the $s$-th iterated
comultiplication $\Delta^s$. The space $\mathcal E^{\m|n}$ has the
{\em monomial basis}
$$K_f :=  v_{f[-m,-m+m_1)} \otimes
 v_{f[-m+m_1,-m+m_1+m_2)} \otimes \cdots  \otimes v_{f[-m_s,-1]} \otimes
 w_{f[1,n]},$$
 where $f$ runs over $\Z_+^{\m|n}$ and $w_{f[1,n]}=w_{f(1)}\wedge\cdots\wedge w_{f(n)}$.

For $i \in I(m|n)$ we define $d_i \in \Z^{m|n}$ by $j \mapsto -
\text{sgn} (i) \delta_{ij}$. For $f, g\in\Z^{m|n}$, we write $f
\downarrow g$ if one of the following holds:
\begin{enumerate}
\item $g=f-d_i +d_j$ for some $i<0 <j$ such that $f(i) =f(j)$;
\item $g =f \cdot \tau_{ij}$ for some $i<j<0$ such that $f(i)
>f(j)$; \item $g =f \cdot \tau_{ij}$ for some $0<i<j$ such that
$f(i) <f(j)$.
\end{enumerate}
The {\em super Bruhat ordering} on $\Z^{m|n}$ is defined as
follows: for $f,g\in\Z^{m|n}$, we say that $f\succcurlyeq g$, if
there exists a sequence $f=h_1,   \ldots, h_r=g \in \Z^{m|n}$ such
that $h_1 \downarrow h_2, \cdots, h_{r-1} \downarrow h_r$. It can
also be described cf.~\cite[\S 2b]{Br} by a number of inequalities
in terms of the $\epsilon$-weights on $\Z^{m|n}$, which are
defined by:
\begin{equation}\label{superweight}
\wt^{\epsilon} (f):=\sum_{i\in
I(m|n)}-\text{sgn}(i)\epsilon_{f(i)}, \qquad \text{for } f \in
\Z^{m|n}.
\end{equation}
The super Bruhat ordering on $\Z^{m|n}$ induces a super Bruhat
ordering on the subsets $\Z^{\m|n}_+, \Z^{\m|n}_{++}$, and
$\Z^{\m|\infty}_+$.

For $n\in\N$, the {\em degree of atypicality} (or {\em atypicality
number}) of $f\in\Z^{m|n}$ is defined to be
$$\#f:= \hf\left(m+n -  \sum_{a\in\Z} |(\wt^\epsilon (f), \epsilon_a)|\right).$$
For $f\in\Z^{m|\infty}$, we define $\#f$ to be the degree of
atypicality of the restriction of $f$ to $I(m|n)$ for $n \gg 0$
(which is clearly well-defined).

If $f, g \in  \Z^{\m|n}_+$ are comparable under the super Bruhat
ordering, then $\#f =\#g.$ If $\#f=0$, we say that $f$ is {\em
typical}; otherwise $f$ is {\em atypical}. An element
$\Z^{\m|n}_+$ is minimal in the super Bruhat ordering if and only
if $f$ is typical and $f \cdot\tau_{ij}$ is not conjugate under
the action of $S_{\m|n}$ to an element in $\Z^{\m|n}_+$ whenever
$f(i)>f(j)$ with $i<j <0$.
\subsection{Bases for $\widehat{\mathcal E}^{\m|n}$}\label{aux:base}

Let $n \in \N \cup \infty$. For $d\in\N$ let ${\mathcal
E}^{\m|n}_{\ge -d}$ be the $\mathbb Q(q)$-subspace of $\mathcal
E^{\m|n}$ spanned by $K_f$ with $f(i)\ge -d$, for all $i\in
\{1,\cdots,n\}$. Following \cite{Br, CWZ} we shall denote a certain
topological completion of ${\mathcal E}^{\m|n}$ by
$\widehat{\mathcal E}^{\m|n}$ whose elements may be viewed as
infinite $\mathbb Q(q)$-linear combinations of elements in $\mathcal
E^{\m|n}$, which under the projection onto $\mathcal E^{\m|n}_{\ge
-d}$ are finite sums for all $d\in\N$ (cf.~\cite[\S2-d]{Br}).

We can define a quasi-matrix following \cite[Chap.~24, 27]{Lu},
that extends the bar-involutions on $\mathcal E^{{\bf m}|0}$ and
on $\mathcal E^{0|n}$.  Using this we can then construct a
bar-involution on $\mathcal E^{{\bf m}|n}$. The following
proposition is a variant of \cite[Theorem~2.14, Theorem~3.5]{Br}
and results of Lusztig, and it can be proved similarly.

\begin{proposition}\label{thm:involutionsuper}
Let $n \in \N \cup \infty$. There exists a unique continuous,
anti-linear bar map ${}^-: \widehat{\mathcal E}^{\m|n} \rightarrow
\widehat{\mathcal E}^{\m|n}$ such that
\begin{enumerate}
\item $\overline{K_f} = K_f$, for all $f \in \Z_+^{\m|n}$ minimal
in the super Bruhat ordering.

\item $\overline{X u} = \overline{X} \overline{u}$, for all $X \in
\mathcal U$ and $u \in \widehat{\mathcal E}^{\m|n}$.

\item The bar map is an involution.

\item $\overline{K_f} = K_f + (*)$, where $(*)$ is a (possibly
infinite) $\Z [q,q^{-1}]$-linear combination of $K_g$'s, with $g
\in \Z_+^{\m|n}$ such that $g  \prec f$.
\end{enumerate}
\end{proposition}

The next theorem now follows by standard arguments (cf. \cite{KL,
Lu, Br}).

\begin{theorem}
Let $n \in \N \cup \infty$. There exist unique {\em canonical
basis} $\{U_f\}$ and {\em dual canonical basis} $ \{L_f \}$, where
${f \in \Z_+^{\m|n}}$,  for $\widehat{\mathcal E}^{\m|n}$ such
that
\begin{enumerate}
\item $\overline{U}_f =U_f$ and $\overline{L}_f =L_f$. \item $U_f
\in K_f + \widehat{\sum}_{g \in \Zmn} q \Z [q] K_g$
 and $L_f \in K_f + \widehat{\sum}_{g \in \Zmn} q^{-1} \Z [q^{-1}]
 K_g$.
\item $U_f = K_f + (*)$ and  $L_f = K_f + (**)$, where $(*)$ and
$(**)$ are (possibly infinite) $\Z [q,q^{-1}]$-linear combinations
of $K_g$'s, with $g \in \Z_+^{\m|n}$ such that $g  \prec f$.
\end{enumerate}
\end{theorem}
The $\widehat{\sum}$ here and further denotes a possibly infinite
sum.

Let $n \in \N \cup \infty$. Generalizing \cite{Br}, we define the
Brundan-Kazhdan-Lusztig polynomials $u_{g,f}(q) \in \Z[q],
\ell_{g,f}(q) \in \Z[q^{-1}]$ associated to $f,g \in \Z_+^{\m|n}$
by
\begin{eqnarray}  \label{eq:superul}
U_f =\sum_{g \in \Z_+^{\m|n}} u_{g,f}(q) K_g, \qquad
L_f =\sum_{g \in \Z_+^{\m|n}} \ell_{g,f}(q) K_g.
\end{eqnarray}
Note that $u_{g,f}(q) =\ell_{g,f}(q)=0$ unless $g \preccurlyeq f$,
$u_{f,f}(q) =\ell_{f,f}(q) =1$, and $u_{g,f}(q)\in q\Z[q]$,
$\ell_{g,f}(q) \in q^{-1} \Z[q^{-1}]$ for $g \ne f$.

\begin{remark}\label{dualUK:super}
By studying a certain symmetric bilinear form on
$\widehat{\mathcal E}^{\m|n}$ such that $\langle L_f, U_{-g\cdot
w_0}\rangle =\delta_{f,g}$ for all $f,g \in \Z_+^{\m|n}$, one can
show (as in \cite[2-i,3-c]{Br} for the special cases for
$\m=(1,\ldots,1)$ or $\m =m$) that
\begin{eqnarray} \label{dualbasis}
K_f =\sum_{g \in \Z_+^{\m|n}} u_{-g\cdot w_0,-f\cdot w_0}(q^{-1})
L_g =\sum_{g \in \Z_+^{\m|n}} \ell_{-g\cdot w_0,-f\cdot
w_0}(q^{-1}) U_g,\quad f\in \Z_+^{\m|n}.
\end{eqnarray}
\end{remark}

\begin{remark} \label{rem:general}
Let $\n =(n_1,\ldots, n_r) \in \N^r$ for $r \ge 1$. One can
generalize readily the bar-involution, the monomial and (dual)
canonical bases to the more general space $\mathcal E^{\m|\n}:=
\otimes_{a=1}^s\Lambda^{m_a} \mathbb V \bigotimes
\otimes_{b=1}^r\Lambda^{n_b} \mathbb W.$ The bases are naturally
parameterized by a set denoted by $\Z_+^{\m|\n}$, which is an
obvious generalization of $\Z_+^{\m|n}$.
\end{remark}

\subsection{The truncation map}
\label{trunct-bar}

Let $n$ be finite. Denote by ${\mathcal E}^{\m|n}_+$ the subspace
of ${\mathcal E}^{\m|n}$ spanned by $K_f$, for $f \in
\Z_{++}^{\m|n}$. For $\infty \geq n'>n$, and $f\in\Z^{\m|n'}_{+}$
(respectively $f \in\Z^{\m+n'}_{+}$), we define $f^{(n)}
\in\Z^{\m|n}_{+}$ (respectively $f^{(n)} \in\Z^{\m+n}_{+}$) to be
the restriction of $f$ to $I(m|n)$. We define the {\em truncation
map} to be the $\mathbb Q(q)$-linear map
$${\mathfrak{Tr}}_{n',n}:  \widehat{\mathcal E}^{\m|n'}_+
\longrightarrow \widehat{\mathcal E}^{\m|n}_+,$$
which sends $K_f$ to $K_{f^{(n)}}$ if $f(i) =i $, for all $i \ge
n+1$, and to $0$ otherwise. We will write ${\mathfrak{Tr}}_{n',n}$
as ${\mathfrak{Tr}}_{n}$ when no ambiguity arises.

\begin{proposition} \label{commutativity2}
For $\infty \geq n'>n$, the truncation map ${\mathfrak{Tr}}_{n',
n}: \widehat{{\mathcal E}}^{\m|n'}_+\rightarrow\widehat{{\mathcal
E}}^{\m|n}_+$ commutes with the bar-involution.
\end{proposition}

\begin{proof}
It suffices to prove the case $n'=n+1$. The proof of
\cite[Proposition~2.8]{CWZ} for the special case when $\m =m$
using the quasi $R$-matrix carries over to this general situation.
\end{proof}

\begin{corollary}\label{aux43} Let $\infty \geq n'>n$.
\begin{enumerate}
\item $\{U_f \}_{f \in \Z_{++}^{\m|n}}$ (respectively $\{L_f \}_{f
\in \Z_{++}^{\m|n}}$) is a basis for $\widehat{\mathcal
E}^{\m|n}_+$.

\item ${\mathfrak{Tr}}_{n',n}$ sends $U_f \in \widehat{\mathcal
E}^{\m|n'}_+$ to $U_{f^{(n)}}$ if $f(i) =i $ for all $i \ge n+1$,
and to $0$ otherwise.

\item ${\mathfrak{Tr}}_{n',n}$ sends $L_f \in \widehat{\mathcal
E}^{\m|n'}_+$ to $L_{f^{(n)}}$ if $f(i) =i $ for all $i \ge n+1$,
and to $0$ otherwise.

\item For $f, g \in \Z_{++}^{\m|n'}$ such that $f(i) =g(i) =i $
for all $i \ge n+1$, we have
$$u_{g,f} (q) =u_{g^{(n)}, f^{(n)}} (q), \qquad \ell_{g,f} (q) =
\ell_{g^{(n)}, f^{(n)}} (q).$$
\end{enumerate}
\end{corollary}
\section{The parabolic Brundan-Kazhdan-Lusztig theory for $\mathfrak g\mathfrak l(m|n)$}
\label{sec:superRep}

\subsection{The category $\FPn$}

For $m,n \in \N$ the Lie superalgebra $\mathfrak g =\glmn$ is
generated by the elementary matrices $e_{ij}$, where $i,j \in
I(m|n)$. For $i \in I(m|n)$, let $\bar{i}=\bar{0}$ if $i<0$ and
$\bar{i}=\bar{1}$ if $i>0$.
%The subalgebra $\mathfrak g_{\bar{0}}$ of $\mathfrak g$ is
%generated by those $e_{ij}$ such that $\bar{i} +\bar{j} =\bar{0}$
%and it is isomorphic to ${\mathfrak g\mathfrak l}(m) \oplus
%{\mathfrak g\mathfrak l}(n)$.
Let $\mathfrak h$ be the standard Cartan subalgebra of $\mathfrak
g$ consisting of the diagonal matrices, $\mathfrak b$ the standard
Borel subalgebra of the upper triangular matrices, and $\Delta^+$
the set of positive roots for $\mathfrak g$. By means of the
natural inclusion ${\mathfrak{gl}}(m|n) \subseteq
{\mathfrak{gl}}(m|n+1)$ via $I(m|n) \subseteq I(m|n+1)$, we let
${\mathfrak{gl}}(m|\infty):=\lim\limits_{\stackrel{\longrightarrow}{n}}
{\mathfrak{gl}}(m|n)$.

Recall that $\m =(m_1,\ldots,m_s)$ with $\sum_i m_i =m$. Consider
the Levi subalgebra $\mathfrak l :=\mathfrak{gl}(m_1)\oplus \cdots
\oplus \mathfrak{gl}(m_s) \oplus \mathfrak{gl}(n)$ and the
corresponding parabolic subalgebra $\mathfrak p :=\mathfrak l
+\mathfrak b$ of $\mathfrak g$. (We shall occasionally write
$\mathfrak p_n$ if we need to keep track of $n$.)

Let $\{\delta_i| i \in I(m|n)\}$ be the basis of $\mathfrak h^*$
dual to $\{e_{ii}\vert i \in I(m|n)\}$. Let $X_{m|n}$ be the set
of integral weights $\la=\sum_{i \in I(m|n)} \la_i \delta_i$,
$\la_i \in\Z$. A symmetric bilinear form on $\mathfrak h^*$ is
defined by
$$(\delta_i|\delta_j)=-{\rm sgn}(i)\delta_{ij},\qquad i,j \in I(m|n).$$
Define
\begin{align*}
X^+_{\m|n} &:= \{\la \in X_{m|n} \mid
  \la_{-m} \ge \cdots \ge\la_{-m+m_1-1}, \\
& \qquad\quad \la_{-m+m_1} \ge \cdots \ge \la_{-m+m_1+m_2-1},\\
& \qquad\quad \cdots, \la_{-m_s} \ge \cdots \ge \la_{-1}, \la_1
\ge \cdots \ge \la_n\},   \\
X_{\m|n}^{++} &:= \{\la \in X^+_{\m|n} \mid  \la_n \geq 0 \}.
%,\\
%X^+_{\m|\infty} &:= \{\la \in X_{m|n} \mid
%  \la_{-m} \ge \cdots \ge\la_{-m+m_1-1}, \\
%& \qquad\quad \la_{-m+m_1} \ge \cdots \ge \la_{-m+m_1+m_2-1},\\
%& \qquad\quad \cdots, \la_{-m_s} \ge \cdots \ge \la_{-1}, \la_1
%=\la_2 = \cdots =0\}.
\end{align*}

We may regard an element $\la$ in $X^{++}_{\m|n}$ as an element in
$X^{++}_{\m|n'}$ for $n'>n$ by adjoining zeros, i.e. letting
$\la_{i} =0$ for $n' \ge i \ge n+1$. Let
$$X^{++}_{\m|\infty} \equiv X^{+}_{\m|\infty}
:=\lim\limits_{\stackrel{\longrightarrow}{n}} X^{++}_{\m|n}.
$$
For $n \in \N \cup \infty$ define
$$\rho = -\sum_{i \in I(m|n)} i \delta_i.$$
Define a bijection
\begin{eqnarray}
 X_{m|n} \longrightarrow \Z^{m|n}, \qquad \la \mapsto f_\la,
\end{eqnarray}
where $f_\la \in \Z^{m|n}$ is given by $f_\la (i) = (\la +\rho |
\delta_i)$ for all $i \in I(m|n)$.
This map induces bijections $X^{+}_{\m|n} \rightarrow \Zmn$ and
$X^{++}_{\m|n} \rightarrow \Z_{++}^{\m|n}$. Using this bijection
we define the notions such as the degree of atypicality,
$\epsilon$-weight, partial order $\preccurlyeq$, et cetera, for
elements in $X^+_{\m|n}$ by requiring them to be compatible with
those defined for elements in $\Zmn$.

For $\la \in X_{m|n}$, we define the {\em parabolic Verma module}
to be
$$
K_n(\la) :=U(\mathfrak g) \otimes_{U (\mathfrak p)} L^0_n(\la),$$
where $L^0_n(\la)$ is the irreducible $\mathfrak l$-module of
highest weight $\la$ extended trivially to a $\mathfrak p$-module.
The irreducible quotient $\mathfrak g$-module of $K_n(\la)$ is
denoted by $L_n(\la)$. Let $[M: L_{n}(\la)]$ denote the
multiplicity of the composition factor $L_{n}(\la)$ in a
$\mathfrak g\mathfrak l(m|n)$-module $M$. When $n=\infty$ we will
make it a convention to drop the subscript $n$ in $K_n(\la),
L_n(\la)$ et cetera.

For $n\in \mathbb N $, $\FPn$ is the category of finitely
generated $\mathfrak{gl}(m|n)$-modules $M$, with $M$ semisimple
over $\mathfrak l$, locally finite over $\mathfrak p$, and
\begin{equation*}
M=\bigoplus_{\g\in X_{m|n}}M_\g,
\end{equation*}
where as usual $M_\g$ denotes the $\g$-weight space of $M$ with
respect to $\mathfrak h$. Note that any object in $\FPn$, when
regarded as a module over its even subalgebra, has finite length by
results of the classical category $\mathcal O$, and hence it has
finite length as well. Denote by $\text{Hom}_{\m|n}$ the
$\text{Hom}$ space in the category $\FPn$. We twist the standard
$\mathfrak g$-module structure on the graded dual $M^*$ of such an
$M$ with the automorphism given by the negative supertranspose on
$\mathfrak g$, and denote the resulting $\mathfrak g$-module by
$M^\tau$. We denote by $\FPPn$ the full subcategory of $\FPn$ which
consists of modules whose composition factors are of the form
$L_n(\la)$ for $\la \in X_{\m|n}^{++}$. We let $\mathcal
O^{++}_{\m|\infty}$ be the category of finitely generated
$\mathfrak{gl}(m|\infty)$-modules that are $\mathfrak l$-semisimple,
locally finite over $ \mathfrak{p}_N$ $\subset \mathfrak{gl}(m|N)$
for all finite $N$, and such that the composition factors are of the
form $L(\la)$ for $\la \in X_{\m|\infty}^{+}$.

\subsection{The truncation functor}
\label{subsec:truncat}

Let $\wt(v)$ denote the weight (or $\delta$-weight) of a weight
vector $v$ in a $\mathfrak g\mathfrak l(m|n)$-module.
\begin{definition}\label{def:trunc}
For $n<n' \le \infty$, the {\em truncation functor}
${\mathfrak{tr}}_{n',n}: \mathcal O_{\m|n'}^{++} \longrightarrow
\FPPn$ is the exact functor which sends an object $M$ to
$${\mathfrak{tr}}_{n',n}(M) := \text{span } \{ v \in M \mid (\wt(v)| \delta_{k}) =
0, \text{ for all }n+1 \le k \le n'\}.$$ When $n'$ is clear from
the context we will also write ${\mathfrak{tr}}_{n}$ for
$\mathfrak{tr}_{n',n}$. (It is easily checked that
${\mathfrak{tr}}_{n',n}(M) \in \FPPn$.)
\end{definition}
We have a system  of categories $\FPPn$ with a compatible sequence
of functors ${\mathfrak{tr}}_{n',n}$ in the sense that
${\mathfrak{tr}}_{n'',n} = {\mathfrak{tr}}_{n',n} \circ
{\mathfrak{tr}}_{n'',n'}$ for $n''>n'>n$.

We have the natural inclusions
${\mathfrak g\mathfrak l}(m|n) \subset {\mathfrak g\mathfrak
l}(m|n+1)$. The following is a variant of \cite[Lemma~3.5]{CWZ}
and can be proved similarly.
\begin{lemma} Let $Y =L$ or $K$.
We have the natural inclusions of $\glmn$-modules:
$Y_n(\la) \subseteq Y_{n+1} (\la) \text{ for }\la \in
X^{++}_{\m|n}.$ Furthermore,
%$Y_n(\la)$ is the $\gl(m|n)$-module generated by a
%highest weight vector and
$\mathfrak{tr}_{n+1,n}(Y_{n+1}(\la)) =Y_n(\la)$.
\end{lemma}

It follows that $\cup_n K_n(\la)$ and $\cup_n L_n(\la)$ are
naturally $\glsuper$-modules. They are direct limits of
$\{K_n(\la)\}$ and $\{L_n(\la)\}$ and isomorphic to $K(\la)$ and
$L(\la)$, respectively. Similarly $\cup_n L_n^0(\la) \cong
L^0(\la)$ as $\mathfrak l$-modules.

\begin{corollary} \label{lem:trunK}
For $\la \in X^{+}_{\m|n'}, n <n' \le \infty,$ and $Y =L$ or $K$,
we have
\begin{equation*}{\mathfrak{tr}}_{n',n}(Y_{n'} (\la)) =
 \left\{ \begin{array}{rr}
  Y_n(\la), & \quad \text{if }  \la_{i} =0 \;\forall i>n, \\
  0, & \quad \text{otherwise.}
  \end{array} \right.
 \end{equation*}
 \end{corollary}

\begin{lemma}  \label{lem:domin}
Let $\la \in X^{++}_{\m|n}$ and $\mu \in X_{\m|n}^+$ be such that
$\mu \preccurlyeq \la$. Then $\mu \in X^{++}_{\m|n}$.
 \end{lemma}
\begin{proof}
Recall that the super Bruhat ordering $\succcurlyeq$ is defined to
be the transitive closure of the three cases of dominance
$f\downarrow g$ in Subsection~\ref{subsec:superFock}, where only in
the first case therein the set $\{f(i)\}_{1\le i \le n}$ will be
changed. More precisely, one particular $f(i)$ involved in an
atypical pair is replaced by some smaller integer.

Thus, thanks to $\la\succcurlyeq \mu$, $\{f_\mu(i)\}_{1\le i \le
n}$ is obtained by consecutively lowering the values
$\{f_\la(i)\}_{1\le i \le n}$ (which are involved in atypical
pairs), whence $\mu \in X^{++}_{\m|n}$.
\end{proof}

Given $M \in \FPn$, denote by $[M]$ the corresponding element in
the Grothendieck group $G(\FPn)$ of the category $\FPn$. Corollary
\ref{lem:trunK} and the exactness of the truncation functor
${\mathfrak{tr}}_{n',n}$ implies the following.

\begin{proposition}   \label{samemult}
For $\la, \mu \in X^{++}_{\m|n}$ and $n' \ge n$, we also regard
$\la, \mu \in X^{++}_{\m|n'}$ by adjoining zeros. Then,
$[K_n(\la): L_n(\mu)] =[K_{n'}(\la): L_{n'}(\mu)].$
\end{proposition}

Given $\la \in X^{+}_{\m|k}$, we denote by $\mathfrak J_{k}(\la)$
the set of the highest weights of the composition factors of
$K_{k}(\la)$ and by $r_{k}(\la)$ the length of a composition
series of $K_{k}(\la)$. Clearly, there exists $\,n(\la)\in\N$ such
that the degree of atypicality $\# \la$ (where we regard $\la \in
X^{++}_{\m|n}$ by adjoining zeros) is independent of $n$ for
$\infty \geq n \ge n(\la)$.

\begin{proposition} \label{series}
\begin{enumerate}
 \item
The $r_{n}(\la)$ and $\mathfrak J_n(\la)$ (with the tail of zeros
in a weight ignored) are independent of $n \ge n(\la)$.
Furthermore, for $n'\geq n \ge n(\la)$ the truncation functor
$\mathfrak{tr}_{n',n}$ maps bijectively the set of Jordan-H\"older
series for $K_{n'}(\la)$ to the set of Jordan-H\"older series for
$K_n(\la)$.
 \item
The parabolic Verma module $K(\la)$ for $\la \in
X^{+}_{\m|\infty}$ has a finite composition series, whose
composition factors are of the form $L(\mu)$ with $\mu\in
X^{+}_{\m|\infty}$, and hence, $K(\la)\in\Fi$. Furthermore,
$[K(\la): L(\mu)]=[K_n(\la): L_n(\mu)].$
\end{enumerate}
\end{proposition}

\begin{proof}
(1) Let $n \ge n(\la)$. $[K_n(\la): L_n(\mu)] \neq 0$ for some
$\mu$ implies $\mu \preccurlyeq \la$. Thus we have $\mu \in
X^{++}_{\m|n}$ and actually $\mu \in X^{++}_{\m|n(\la)}$ by the
proof of Lemma~\ref{lem:domin}, where indeed $f_\mu(i) =f_\la(i)$
for $i>n(\la)$. Hence the first statement follows by
Proposition~\ref{samemult}. Now the second statement follows from
the first one and Lemma~\ref{lem:trunK} using the same argument as
for \cite[Lemma~3.8]{CWZ}.

(2) follows from the special case of (1) with $n'=\infty$.
\end{proof}

\subsection{The tilting modules}
\label{subsec:tilt}

Throughout this subsection we assume that $n$ is finite. An object
$M \in \FPn$ is said to have a {\em Verma flag} (respectively, a
{\em dual Verma flag}) if it has a filtration of $\mathfrak
g\mathfrak l (m|n)$-modules:
$$0 =M_0 \subseteq \cdots \subseteq M_r =M$$
such that each $M_i/M_{i-1}$ is isomorphic to a parabolic Verma
module $K_n(\la^i)$ (respectively, $K_n(\la^i)^\tau$) for some
$\la^i \in X_{\m|n}^+$. We define $(M :K_n(\mu))$ for $\mu\in
X_{\m|n}^+$ to be the number of subquotients of a Verma flag of
$M$ that are isomorphic to $K_n(\mu)$.
%
%\begin{definition} \label{def:tilt}
The {\em tilting module} associated to $\la\in X_{m|n}^+$ in the
category $\FPn$ is an indecomposable $\glmn$-module $U_n(\la)$
such that
%\begin{itemize}
%\item[(1)]
$U_n(\la)$ has a Verma flag with $K_n(\la)$ at the bottom, and
%\item[(2)]
$\text{Ext}^1(K_n(\mu),U_n(\la))=0$ for all $\mu\in X^+_{\m|n}$.
%\end{itemize}
%\end{definition}
By a parabolic version of \cite{Br2} as in Soergel \cite{So2} for
the usual semisimple Lie algebras, the tilting module $U_n(\la)$ in
the category $\FPn$ exists and is unique. Following \cite{Br2, So2},
the projective cover $P_n(\la)$ of $L_n(\la)$ exists for each
$\la\in X^+_{\m|n}$ and admits a finite Verma flag. The following is
a synthesis of standard results (see \cite{Jan, Br2}) adapted to our
particular setup.

\begin{proposition} \label{flagExt}
\begin{enumerate}
 \item
Let $M$ be a module with a finite Verma flag and $N$ be a module
with a finite dual Verma flag. Then, ${\rm Ext}^i(M,N)=0$ for all
$i>0$.
 \item
Let $N \in \FPn$. Then the following statements are equivalent:
\begin{itemize}
 \item[(a)]
 $N$ has a dual Verma flag;
 \item[(b)]
 ${\rm Ext}^i(K_n(\la),N)=0$ for all $\la \in  X^+_{\m|n}$ and
 all $i>0$;
 \item[(c)]
 ${\rm Ext}^1(K_n(\la),N)=0$ for all $\la \in  X^+_{\m|n}$.
\end{itemize}
 \item
 A tilting module in $\FPn$ has a finite dual Verma flag.
\end{enumerate}
\end{proposition}

\begin{proof}
Part (2) can be proved using (1) exactly as for
\cite[Proposition~4.16]{Jan}. Part (3) follows from (2) (also see
\cite{Br2}).

So it remains to prove (1). Using an induction on the Verma flag
length on $M$
%it suffices to show that ${\rm Ext}^i(K_n(\la),N)=0$ for all $\la$.
and then an induction on the dual Verma flag length on $N$, it
suffices to show that ${\rm Ext}^i(K_n(\la),K_n(\mu)^\tau)=0$ for
all $\la, \mu$ and $i \ge 1$.

As in \cite[Lemma~3.6~(iii)]{Br2}, we have ${\rm
Ext}^i(K_n(\la),K_n(\mu)^\tau)=0$ with $i=1$.
%
%Consider an extension of $\mathfrak g$-modules
%\begin{equation*}
%0\rightarrow K_n(\mu)^\tau\rightarrow T\rightarrow
%K_n(\la)\rightarrow 0.
%\end{equation*}
%By applying the exact functor $\tau$ if necessary, we may assume
%without loss of generality that $\la\not\in\mu-Q_+$ with $Q_+ =\{0
%\neq \sum_{\alpha\in\Delta_+}m_\alpha\alpha \mid m_\alpha\ge 0\}$.
%If $\la\not=\mu$, then $T_\la$ is one-dimensional and is generated
%by a singular vector. Thus there exists a map $K_n(\la)\rightarrow
%T$, and so the sequence splits. If $\mu=\la$, then $T_\la$ is
%two-dimensional and consists of highest weight vectors, and again
%the sequence splits.
The ${\rm Ext}^i$ vanishing for $i>1$ follows by a standard
induction argument, which we sketch below for the convenience of the
reader. We have an exact sequence
\begin{equation*}
0\rightarrow K\rightarrow P_n(\la)\rightarrow K_n(\la)\rightarrow 0
\end{equation*}
where  $K$ has a finite Verma flag. We get a long exact sequence
\begin{equation*}
\leftarrow {\rm Ext}^{i+1}(P_n(\la),K_n(\mu)^\tau)\leftarrow {\rm
Ext}^{i+1}(K_n(\la),K_n(\mu)^\tau)\leftarrow {\rm
Ext}^i(K,K_n(\mu)^\tau) \leftarrow \cdots.
\end{equation*}
Note that ${\rm Ext}^{i+1}(P_n(\la),K_n(\mu)^\tau)=0$, since
$P_n(\la)$ is projective. By inductive assumption, $ {\rm
Ext}^i(K_n(\nu),K_n(\mu)^\tau) =0$ for all $\nu$, and thus ${\rm
Ext}^i(K,K_n(\mu)^\tau)=0$ by induction on the Verma flag length of
$K$. Hence, ${\rm Ext}^{i+1}(K_n(\la),K_n(\mu)^\tau)=0$.
\end{proof}

%\begin{corollary}  \label{dualVerma}
%A tilting module in $\FPn$ has a finite dual Verma flag.
%\end{corollary}

\begin{corollary}\label{tiltingdual}
We have $U_n(\la)\cong U_n(\la)^\tau$.
\end{corollary}

\begin{proof}
We have ${\rm Ext}^1(U_n(\la),K_n(\mu)^\tau)=0$ by
Proposition~\ref{flagExt}~(1), and hence by applying the functor
$\tau$, ${\rm Ext}^1(K_n(\mu),U_n(\mu)^\tau)=0$. By the construction
of tilting modules \cite{So2}, ${\rm Hom}_{{\bf
m}|n}(K_n(\mu),U_n(\la))=0$, for $\mu\succ\la$, and ${\rm Hom}_{{\bf
m}|n}(K_n(\la),U_n(\la))=1$. Thus there are no weights in
$U_n(\la)^\tau$ greater than $\la$, which appears with multiplicity
one. Now $U_n(\la)^\tau$ also has a Verma flag by
Proposition~\ref{flagExt}.  Thus $U_n(\la)^\tau\cong U_n(\la)$ by
uniqueness of tilting modules.
\end{proof}

\subsection{A parabolic version of the Brundan conjecture}
\label{subsec:BrConj}

The same arguments as in \cite{Br2, So2} give us the following:
\begin{eqnarray} \label{eq:tiltduality}
(U_n(\la): K_n(\mu)) =[K_n(-w_0\mu- 2\rho+2\rho_{\mathfrak l}):
L_n(-w_0\la- 2\rho+2\rho_{\mathfrak l})],
\end{eqnarray}
where we recall that $w_0$ is the longest element in the Weyl
group $S_{\m|n}$ of the Levi subalgebra $\mathfrak l$, and
$\rho_{\mathfrak l}$ is half the sum of positive roots of
$\mathfrak l$.

It is well known that each $\la \in X^+_{\m|n}$ (or more generally
$\la \in \mathfrak h^*$) gives rise to a central character
$\chi_\la$. There is a neat characterization of central characters
in terms of $\epsilon$-weights \cite[Lemma~4.18]{Br}: $\chi_\la
=\chi_\mu$ for $\la, \mu \in X^+_{\m|n}$ if and only if
$\wt^\epsilon (f_\la)=\wt^\epsilon (f_\mu).$ It follows that the
category $\FPn$ has a ``block" decomposition $\FPn =\sum_{\g \in
P} \mathcal O^+_\g$.

Let $V$ be the natural $\mathfrak g\mathfrak l (m|n)$-module and
$V^*$ its dual. For $a\in \Z, r \ge 1$ we define the translation
functors $E_a^{(r)}, F_a^{(r)}: \FPn \longrightarrow \FPn$ by
sending $M \in \mathcal O^+_\g$ to
\begin{eqnarray}  \label{eq:transl}
 F_a^{(r)} M &:=& \text{pr}_{\g - r(\epsilon_a
 -\epsilon_{a+1})} (S^rV \otimes M),\nonumber \\
 E_a^{(r)} M &:=& \text{pr}_{\g + r(\epsilon_a
 -\epsilon_{a+1})} (S^rV^* \otimes M).
\end{eqnarray}
By convention, set $F_a =F_a^{(1)}, E_a =E_a^{(1)}.$ Let
$\FnDelta$ be the full subcategory of $\FPn$ consisting of all
modules with Verma flags. Let $G(\FnDelta)_{\mathbb
Q}:=G(\FnDelta)\otimes_\Z{\mathbb Q}$ and let $\mathcal
E^{\m|n}|_{q=1}$ be the specialization of $\mathcal E^{\m|n}$ as
$q\to 1$. Denote the $q\to 1$ specialization of $U_f, K_f$ by
$U_f(1), K_f(1)$ et cetera.

\begin{theorem} \label{superChe=tran}
Let $n \in \N$.
\begin{enumerate}
\item Sending the Chevalley generators $E_a^{(r)}, F_a^{(r)}$
($a\in \Z, r \ge 1$) to the translation functors $E_a^{(r)},
F_a^{(r)}$ defines a $\mathcal U_{q=1}$-module structure on
$G(\FnDelta)_{\mathbb Q}$.

\item The linear map $i: G(\FnDelta)_{\mathbb Q} \rightarrow
\mathcal E^{\m|n}|_{q=1}$, which sends $[K_n(\la)]$ to
$K_{f_\la}(1)$, for each $\la \in X^+_{\m|n}$, is an isomorphism of
$\mathcal U_{q=1}$-modules.
\end{enumerate}
\end{theorem}

\begin{proof}
This is a straightforward generalization of \cite[Theorems 4.28,
4.29]{Br}, and it can be proved similarly.
\end{proof}

The following is a parabolic version of \cite[Conjecture~4.32]{Br}.

\begin{conjecture} \label{conj:BrKL} [Parabolic
Brundan-Kazhdan-Lusztig Conjecture]
 Let $n \in \N\cup \infty$. The map $i: G(\FnDelta)_{\mathbb Q}
\rightarrow \mathcal E^{\m|n}|_{q=1}$ sends $[U_n(\la)]$ to
$U_{f_\la}(1)$ for each $\la \in X^+_{\m|n}$.
 \linebreak
(The case for $n =\infty$ will be clarified and made plausible by
Theorem~\ref{tiltInf} below.)
\end{conjecture}

Conjecture~\ref{conj:BrKL} can be equivalently reformulated as
either of the following conjectural identities, in light of
(\ref{dualbasis}), (\ref{eq:tiltduality}), and
Theorem~\ref{superChe=tran}: for $\la, \mu \in X^+_{\m|n}$,
\begin{eqnarray*}
(U_n(\la): K_n(\mu)) &=& u_{\mu,\la}(1),  \\
{[}K_n(\la): L_n(\mu)] &=& u_{-w_0\la- 2\rho+2\rho_{\mathfrak l}, -w_0\mu-2\rho+2\rho_{\mathfrak l}}(1),  \\
\text{ch} L_n(\la) &=& \sum_{\mu \in X^+_{\m|n}}
\ell_{\mu,\la}(1)\, \text{ch} K_n(\mu).
\end{eqnarray*}
We note that the validity of \cite[Conjecture~4.32]{Br} would imply
Conjecture \ref{conj:BrKL}.

\begin{remark} \label{rem:genO}
Let $\n =(n_1,\ldots, n_r) \in \N^r$ with $n=\sum_{b=1}^r n_b$.
One can formulate the more general category $\mathcal O_+^{\m|\n}$
of $\gl(m|n)$-modules which are semisimple over $\oplus_{a=1}^s
\gl(m_a) \oplus \oplus_{b=1}^r\gl(n_b)$. All the statements on
tilting modules and the Brundan-Kazhdan-Lusztig Conjecture in
Subsections~\ref{subsec:tilt} and \ref{subsec:BrConj} can be
readily generalized to this more general setup (cf.
Remark~\ref{rem:general}). Brundan's conjecture \cite{Br} was
formulated for the full category $\mathcal O$, i.e. when all $m_a$
and $n_b$ are equal to $1$.

On the other hand, the BKL conjecture on the irreducible
characters in any parabolic category would follow from the
validity of the corresponding Brundan's conjecture for the full
category $\mathcal O$ (using the same argument as for the usual
Lie algebras of type $A$).
\end{remark}
\subsection{Tilting modules with $n$ varied}
\label{subsec:inftytilt}

\begin{proposition}\label{nestedtilt3}
For $\la\in X^{++}_{m|n+1}$ the truncation functor
${\mathfrak{tr}}_n$ sends $U_{n+1} (\la)$ to $ U_n(\la)$ if
$(\la|\delta_{n+1})=0$, and to $0$ otherwise.
\end{proposition}

\begin{proof}
By the construction of tilting modules (cf. \cite{So2, Br2}),
$U_{n+1}(\la)$ has a Verma flag with subquotients isomorphic to
$K_{n+1}(\mu)$ with $\mu \preccurlyeq \la$. If
$(\la|\delta_{n+1})>0$, then $(\mu|\delta_{n+1})>0$ and thus
${\mathfrak{tr}}_n(U_{n+1}(\la))=0$ by Lemma~\ref{lem:trunK}.

Thanks to Lemma~\ref{lem:domin}, the truncation functor
${\mathfrak{tr}}_n$ preserves Verma flags. It follows from the
commutativity of $\tau$ with ${\mathfrak{tr}}_n$ and
Proposition~\ref{flagExt} that ${\mathfrak{tr}}_n$ also preserves
the dual Verma flags. By Proposition~\ref{flagExt}, ${\rm
Ext}^1(K_n(\mu), {\mathfrak{tr}_n}(U_{n+1} (\la))=0$. If
$(\la|\delta_{n+1})=0$, then ${\mathfrak{tr}_n}(K_{n+1} (\la))=K_{n}
(\la)$ and clearly $K_n(\la)$ sits at the bottom of
${\mathfrak{tr}_n}(U_{n+1} (\la))$.

To show that ${\mathfrak{tr}_n}(U_{n+1}(\la)) =U_n(\la)$, it remains
to show that ${\mathfrak{tr}_n}(U_{n+1}(\la))$ is indecomposable.
Indeed, this follows by the same argument for
\cite[Proposition~1.5]{Don} with the help of
Proposition~\ref{flagExt}. We recall here that the counterpart in
our setup of (\cite[Proposition~1.5]{Don} states that
$\text{Hom}_{\m|n+1}(M,N) \rightarrow \text{Hom}_{\m|n}
(\mathfrak{tr}_nM,\mathfrak{tr}_nN)$ is surjective, for $M$
(respectively $N$) with a finite Verma (respectively dual Verma)
flag. Its proof is elementary and uses induction on the (dual) Verma
length, Lemma~\ref{lem:domin}, and the standard fact that
\begin{eqnarray}  \label{eq:onezero}
\text{Hom}_{\m|n} (K_n(\la),K_n(\mu)^\tau) \cong \delta_{\la,\mu}
\C.
\end{eqnarray}
Thus, $\text{End}_{\m|n} (\mathfrak{tr}_n U_{n+1}(\la))$, as a
quotient of the local $\C$-algebra
$\text{End}_{\m|n+1}(U_{n+1}(\la))$, is local. This implies that
${\mathfrak{tr}_n}(U_{n+1}(\la))$ is indecomposable.
\end{proof}

\begin{proposition}   \label{truntilt}
For $\la, \mu \in X^{++}_{\m|n}$ and $n' \ge n$, we also regard
$\la, \mu \in X^{++}_{\m|n'}$ by adjoining zeros. Then,
$(U_n(\la): K_n(\mu)) =(U_{n'}(\la): K_{n'}(\mu)).$
\end{proposition}
\begin{proof}
The proof is similar to the proof for Proposition~\ref{samemult},
now with the help of Lemma~\ref{lem:domin} and
Proposition~\ref{nestedtilt3}.
\end{proof}

\begin{theorem} \label{tiltInf}
Let $\la\in X^{+}_{\m|\infty}$.
\begin{enumerate}
\item There exists a unique (up to isomorphism) tilting module
$U(\la)$ in $\Fi$ with $K(\la)$ sitting at the bottom of a Verma
flag. Moreover, $U(\la) = \cup_n U_n(\la)$.

\item The functor ${\mathfrak{tr}}_n$ sends $U (\la)$ to
 $ U_n(\la)$ if $(\la|\delta_{n+1})=0$ and to $0$ otherwise.

\item We have $(U(\la):K(\mu))=(U_n(\la):K_n(\mu))$ for $n\gg0$.

\item The Verma flag length for $U(\la)$ and $U_n(\la)$ for $n\gg
0$ is the same (and finite).
\end{enumerate}
\end{theorem}

\begin{proof}
We define $U(\la)$ to be $\cup_n U_n(\la)$. The same proof for
\cite[Theorem~3.16]{CWZ} applies here to prove (1) and (2), with
the help of Proposition~\ref{truntilt} above. (3) and (4) follow
by an argument similar to the proof of Proposition~\ref{series}.
\end{proof}

\begin{remark}
Conjecture~\ref{conj:BrKL} as $n$ varies is compatible with the
properties of truncation maps and the truncation functors (cf.
Corollary~\ref{aux43} and Proposition~\ref{nestedtilt3}).
\end{remark}

\section{Kazhdan-Lusztig theory for $\mathfrak g\mathfrak l(m+n)$ revisited and super duality}
\label{sec:FockredKL}

\subsection{Kazhdan-Lusztig polynomials and canonical basis for
${\mathcal E}^{\m+n}$} \label{sec:bar}

In this subsection we give a presentation of certain parabolic
Kazhdan-Lusztig polynomials in terms of the Fock space ${\mathcal
E}^{\m+n}$ (compare \cite{FKK, Br, BKl}).

For $n\in\N$ let $\mathcal E^{\m+n}_+$ denote the subspace of
$\mathcal E^{\m+n}$ spanned by elements of the form $\mathcal
K_f$, $f\in\Z^{\m+n}_{++}.$
%.
For $n'>n$ define the {\em truncation map}
$ {\textsf{Tr}}_{n',n}:  {\mathcal E}^{\m+n'}_+ \longrightarrow
{\mathcal E}^{\m+n}_+ $
by sending $\mathcal K_f$ to $\mathcal K_{f^{(n)}}$ if $f(i+1)
=-i$ for all $i \ge n$, and to $0$ otherwise. This gives rise to
${{\textsf{Tr}}_{n}}: \mathcal E^{\m+\infty}\rightarrow \mathcal
E_+^{\m+n}$, for all $n$, which in turn allows us to define a
topological completion $\widehat{\mathcal
E}^{\m+\infty}:=\lim\limits_{\stackrel{\longleftarrow}{n}}\mathcal
E_+^{\m+n}$, similarly as in \cite[\S2-d]{Br}. For a finite $n$
let $\widehat{\mathcal E}^{\m+n}\equiv {\mathcal E}^{\m+n}$.

The following proposition can be established similarly as
\cite[Theorems~2.14 and 3.5]{Br} for the special cases $\m
=(1,\ldots,1)$ or $\m=m$.

\begin{proposition} \label{th:barinf}
Let $n\in \N \cup \infty$. There exists a unique anti-linear bar
map ${}^-: \widehat{\mathcal E}^{\m+n} \rightarrow
\widehat{\mathcal E}^{\m+n}$ such that
\begin{enumerate}
\item $\overline{\mathcal K_f} = \mathcal K_f$, for all $f \in
\Z_+^{\m+n}$ minimal in the Bruhat ordering.

\item $\overline{X u} = \overline{X} \overline{u}$, for all $X \in
\mathcal U$ and $u \in \widehat{\mathcal E}^{\m+n}$.

 \item The bar map is an involution.

\item $\overline{\mathcal K_f} = \mathcal K_f + (*)$, where $(*)$
is a (possibly infinite when $n=\infty$) $\Z [q,q^{-1}]$-linear
combination of $\mathcal K_g$'s with $g \in \Z_+^{\m+n}$ such that
$g < f$ in the Bruhat ordering.
\end{enumerate}
\end{proposition}

The next theorem follows from Proposition~\ref{th:barinf}.

\begin{theorem} \label{th:canonical}
Let $n \in \N \cup \infty$. There exist unique topological bases
$\{\mathcal U_f\},$ $ \{\mathcal L_f \}$, where ${f \in
\Z_+^{\m+n}}$,  for $\widehat{\mathcal E}^{\m+n}$ such that
\begin{enumerate}
\item $\overline{\mathcal U}_f =\mathcal U_f$ and
$\overline{\mathcal L}_f =\mathcal L_f$;

\item $\mathcal U_f \in \mathcal K_f + \widehat{\sum}_{g \in
\Z_+^{\m+n}} q \Z [q] \mathcal K_g$ and $\mathcal L_f \in \mathcal
K_f + \widehat{\sum}_{g \in \Z_+^{\m+n}} q^{-1} \Z [q^{-1}]
 \mathcal K_g$.

\item $\mathcal U_f = \mathcal K_f + (*)$ and  $\mathcal L_f =
\mathcal K_f + (**)$, where $(*)$ and $(**)$ are (possibly
infinite when $n=\infty$) $\Z [q,q^{-1}]$-linear combinations of
$\mathcal K_g$'s with $g \in \Z_+^{\m+n}$ such that $g  <f$. For
$n$ finite, $(*)$ and $(**)$ are always finite sums.
\end{enumerate}
\end{theorem}

We define $\mathfrak u_{g,f}(q) \in \Z[q], \mathfrak l_{g,f}(q)
\in \Z[q^{-1}]$ for $f,g \in \Z_+^{\m+n}$ by
\begin{eqnarray} \label{ulpoly}
\mathcal U_f =\sum_{g \in \Z_+^{\m+n}} \mathfrak u_{g,f}(q)
\mathcal K_g, \qquad
\mathcal L_f =\sum_{g \in \Z_+^{\m+n}} \mathfrak l_{g,f}(q)
\mathcal K_g.
\end{eqnarray}
Note that $\mathfrak u_{g,f}(q) =\mathfrak l_{g,f}(q)=0$ unless $g
\le f$ and $\mathfrak u_{f,f}(q) =\mathfrak l_{f,f}(q) =1$. These
polynomials can be identified as {\em (parabolic) Kazhdan-Lusztig
polynomials} (cf. Theorem~\ref{th:multiCan} below).

\begin{remark}  \label{dualUK}
By the same type of arguments as in \cite[\S3-c]{Br} we can
introduce a symmetric bilinear form $\langle\cdot,\cdot\rangle$ on
${\mathcal E}^{\m+n}$ such that
%for which the canonical and dual canonical
%bases can be shown to be dual bases:
%
$\langle\mathcal L_f,\mathcal U_{-g\cdot w_0}\rangle=\delta_{f,g}$
for $f,g \in \Z_+^{\m+n}$,
which readily implies that the matrices $[\mathfrak{u}_{-f \cdot
w_0,-g \cdot w_0}(q)]$ and $[\mathfrak l_{f ,g}(q^{-1})]$ are
inverses of each other. Equivalently, we have
\begin{equation*}
\mathcal K_f=\sum_{g\in\Z^{\m+n}_+} \mathfrak{u}_{-f\cdot w_0,-g
\cdot w_0}(q^{-1})\mathcal L_g =
\sum_{g\in\Z^{\m+n}_+}\mathfrak{l}_{-f \cdot w_0,-g \cdot
w_0}(q^{-1})\mathcal U_g, \quad  f\in\Z^{\m+n}_+.
\end{equation*}
\end{remark}

\begin{proposition} \label{commutativity-reduct}
\begin{enumerate}
\item The truncation map ${\textsf{Tr}}_{n', n}: {\mathcal
E}^{\m+n'}_+\rightarrow{\mathcal E}^{\m+n}_+$ commutes with the
bar-involution, where $\infty \geq n'>n$.

\item $\textsf{Tr}_{n',n}$ sends $\mathcal U_f$ (respectively
$\mathcal L_f$) to $\mathcal U_{f^{(n)}}$ (respectively $\mathcal
L_{f^{(n)}}$) if $f(i+1) =-i$, for all $i \ge n$, and to $0$
otherwise.

\item For $f, g \in \Z_{++}^{\m+n'}$ such that $f(i+1) =g(i+1)
=-i$ for all $i \ge n$, we have
$$\mathfrak u_{g,f} (q) =\mathfrak u_{g^{(n)}, f^{(n)}} (q),\quad
\mathfrak l_{g,f} (q) =\mathfrak l_{g^{(n)}, f^{(n)}} (q).$$
\end{enumerate}
\end{proposition}

\begin{proof}
Part (1) is proved similarly as \cite[Proposition~4.29]{CWZ}. (2)
and (3) are immediate corollaries.
\end{proof}
\subsection{A Fock space isomorphism and consequences}

\begin{proposition} \label{Fockiso}
\begin{enumerate}
\item There is an isomorphism of $\mathcal U$-modules  $C: \wgv
\rightarrow \wgw$ which sends $|\la \rangle$ to $|\la'_* \rangle$
for each partition $\la$.

\item The map $C$ extends naturally to an isomorphism of $\mathcal
U$-modules
$$\natural: \widehat{\mathcal E}^{\m+\infty}
\stackrel{\cong}{\longrightarrow} \widehat{\mathcal
E}^{\m|\infty},$$
\end{enumerate}
which is compatible with the actions of all divided powers
$E_a^{(s)}, F_a^{(s)}.$
\end{proposition}
\begin{proof}
Part (1) above is \cite[Theorem~6.3]{CWZ}. Recall that ${\mathcal
E}^{\m+\infty} =\otimes_{i=1}^s\Lambda^{m_i}\mathbb V \otimes
\La^\infty \mathbb V$ and ${\mathcal
E}^{\m|\infty}=\otimes_{i=1}^s\Lambda^{m_i}\mathbb V \otimes
\La^\infty \mathbb V^*$. Then $C: \wgv \rightarrow \wgw$ induces a
$\mathcal U$-module isomorphism $\natural =1\otimes C: {\mathcal
E}^{\m+\infty} \stackrel{\cong}{\longrightarrow} {\mathcal
E}^{\m|\infty}$. One can further check that these two topological
completions $\widehat{\mathcal E}^{\m+\infty}$ and
$\widehat{\mathcal E}^{\m|\infty}$  are indeed compatible under
$\natural$.
\end{proof}

Given $\la = \sum_{i \in I(m|\infty)} \la_i\delta'_i \in
X_{\m+\infty}^{+}$ so that by definition $\la^{>0} :=(\la_1,
\la_2, \ldots)$ is a partition. Denoting by $(\la'_1, \la'_2,
\ldots)$ the conjugate partition of $\la^{>0}$, we define a weight
$$\la^{\natural} := \sum_{i=-m}^{-1} \la_i \delta_i
+\sum_{j=1}^\infty \la_j' \delta_j  \in X_{\m|\infty}^{+}.
$$
This actually defines bijections (denoted by $\natural$ by abuse
of notation)
$$X_{\m+\infty}^{+} \stackrel{\natural}{\longleftrightarrow}
X_{\m|\infty}^{+},\qquad\quad
\Z_+^{\m+\infty}\stackrel{\natural}{\longleftrightarrow}\Z_+^{\m|\infty},
$$
when coupling with the two bijections $X_{\m|\infty}^{+}
\leftrightarrow \Z_+^{\m|\infty}$ and $X_{\m+\infty}^{+}
\leftrightarrow \Z_+^{\m+\infty}$. There is a simple combinatorial
description for the bijection
\begin{eqnarray} \label{bij}
\Z_+^{\m+\infty}
\stackrel{\natural}{\longrightarrow}\Z_+^{\m|\infty}, \quad
(f^{<0}|f^{>0}) \mapsto f=(f^{<0}|\Z \backslash f^{>0})
\end{eqnarray}
in light of \cite[Lemma~6.2]{CWZ}, where $f^{>0}$ denotes the
restriction of $f$ to $I(0|\infty)$ and $\Z \backslash f^{>0}$
denotes the complement of $f^{>0}$ in $\Z$.

\begin{lemma}  \label{lem:move}
\begin{enumerate}
 \item
For $f,g \in \Z_+^{\m+\infty}$, $f \ge g$ in the Bruhat ordering
if and only if $f^\natural \succcurlyeq g^\natural$ in the super
Bruhat ordering.

\item A weight $\la \in X_{m+\infty}^{+}$ is minimal in the Bruhat
ordering if and only if $\la^\natural \in \Xmi^{+}$ is minimal in
the super Bruhat ordering.
\end{enumerate}
\end{lemma}

\begin{proof}
(2) is a special case of (1), so let us prove (1).

Denote by $f^+$ the (unique if exists) conjugate in
$\Z_+^{m+\infty}$ of $f\in \Z^{m+\infty}$ under the action of
$S_{\m+\infty}$. The super Bruhat ordering $\succcurlyeq$ on
$\Z_+^{\m|\infty}$ is the transitive closure of the partial order
$g\succcurlyeq f$ given by
\begin{enumerate}
 \item [(i)] $g=(f-d_i +d_j)^+ \text{ for some } i<0 <j \text{ such that } f(i)
=f(j);$

\item[(ii)] $g =(f \cdot \tau_{ij})^+ \text{ for some } i<j<0
\text{ such that } f(i)>f(j)$.
\end{enumerate}

On the other hand, the Bruhat ordering $\geq$ on
$\Z_+^{\m+\infty}$ is the transitive closure of the partial order
$g \geq f$ given by
\begin{enumerate}
 \item [(i')]  $g =(f \cdot \tau_{ij})^+ \text{ for some } i<0<j
\text{ such that } f(i)>f(j);$

\item[(ii')] $g =(f \cdot \tau_{ij})^+ \text{ for some } i<j<0
\text{ such that } f(i)>f(j)$.
\end{enumerate}

Exactly as explained in the proof of \cite[Lemma~6.6]{CWZ} when
$\m=m$, under the explicit bijection $\natural: \Z_+^{\m+\infty}
{\longrightarrow}\Z_+^{\m|\infty}$ given by (\ref{bij}), the
Step~(i) corresponds to Step~(i'). Now clearly the Step (ii)
corresponds to (ii') by (\ref{bij}). This proves (1).
\end{proof}

\begin{theorem} \label{correspondence}
The isomorphism $\natural: \widehat{\mathcal E}^{\m+\infty}
\longrightarrow \widehat{\mathcal E}^{\m|\infty}$ has the
following properties:
\begin{enumerate}

\item $\natural (\mathcal K_f) = K_{f^\natural}$ for each
$f\in\Z_+^{\m+\infty}$;

 \item $\natural$ is compatible
with the bar involutions, i.e., $\natural (\bar{u})
=\overline{\natural (u)}$ for each $u \in \widehat{\mathcal
E}^{\m+\infty}$;

\item $\natural (\mathcal L_f) = L_{f^\natural}$ for each
$f\in\Z_{+}^{\m+\infty}$;

\item $\natural (\mathcal U_f) =U_{f^\natural}$ for each
$f\in\Z_{+}^{\m+\infty}$.

\item For $f,g \in\Z_{+}^{\m+\infty}$, we have $\mathfrak u_{g,f}
(q) = u_{g^\natural,f^\natural} (q),$
and $ \mathfrak l_{g,f} (q) = \ell_{g^\natural,f^\natural} (q). $
\end{enumerate}
\end{theorem}

\begin{proof}
(1) follows from the definitions and Proposition~\ref{Fockiso}. (2)
follows from Proposition~\ref{Fockiso}, Lemma \ref{lem:move}~(2) and
the characterizations of the bar involutions. (3) and (4) follow
from (1), (2), Lemma \ref{lem:move}, and the characterizations of
these bases.
\end{proof}

The following verifies a parabolic version of \cite[Conjecture
2.28]{Br}.

\begin{theorem}  \label{th:BrConj}
\begin{enumerate}
\item  The Brundan-Kazhdan-Lusztig polynomials satisfy the
following positivity: $\mathfrak u_{\mu,\la} (q) \in \N[q],\;
\mathfrak l_{\mu,\la} (-q^{-1}) \in \N[q]$ for all $\la, \mu \in
X_{\m|n}^{+}$.

\item For each $a\in \Z, r\ge 1$, and $f \in \Z_+^{\m|n}$, the
coefficients of $E_a^{(r)} U_f, F_a^{(r)} U_f$ (respectively
$E_a^{(r)} L_f, F_a^{(r)} L_f$) in the expansion in terms of the
canonical basis $\{U_g\}$ (respectively, the dual canonical basis
$\{L_g\}$) lie in $\N[q, q^{-1}]$.
\end{enumerate}
\end{theorem}

\begin{remark} \label{rem:folklore}
Set $n=0$ in Theorem~\ref{th:BrConj}, and we are in the setup of
the Fock space corresponding to usual parabolic category $\mathcal
O_\m^+$ of $\gl(m)$-modules. It is folklore that
Theorem~\ref{th:BrConj}~(2) with $n=0$ should be true and indeed a
proof is known to Lusztig \cite{Lu3}. Theorem~\ref{th:BrConj}~(2)
with $n=0$ would also follow from the graded lifts in the sense of
Beilinson, Ginzburg and Soergel \cite{BGS} of the category
$\mathcal O_\m^+$ and the divided power translation functors
$E_a^{(r)}, F_a^{(r)}$, for $a \in\Z, r \ge 1$. For example, a
complete proof in a special case of such a lift of the divided
powers has been written down by Frenkel, Khovanov and Stroppel
\cite[Theorems~3.6, 5.3]{FKS} (see Remark~5.6 therein for the
general category $\mathcal O$, and the parabolic case should
follow too). We thank Jon Brundan for the reference and
clarification.
\end{remark}

\begin{proof}
(1) It suffices to prove when $n$ is finite. Let us identify the
Kazhdan-Lusztig polynomials for $\glmn$ with the usual
Kazhdan-Lusztig polynomials for $\mathfrak{gl}(m+N)$ for {\em
finite} $n$ and $N$. Given $\la, \mu \in X_{\m|n}^{++}$, we obtain
$\la_\infty \in X_{\m|\infty}^+$ the extension of $\la$ by zeros,
and $\la_\infty^\natural \in X_{\m+\infty}^+$. Write
$\la_\infty^\natural = ((\la_\infty^\natural)^{<0} |
(\la_\infty^\natural)^{>0})$. Assuming the lengths of the
partitions $(\mu_\infty^\natural)^{>0}$ and
$(\la_\infty^\natural)^{>0}$ are no larger than $N$, we have
$\la_\infty^{\natural,(N)}, \mu_\infty^{\natural,(N)} \in
X_{\m+N}^{++}.$ Then,
\begin{eqnarray*}
 u_{\mu,\la}(q)
 &=& u_{\mu_\infty,\la_\infty}(q)
 =\mathfrak u_{\mu_\infty^\natural,\la_\infty^\natural}(q)
 = \mathfrak u_{\mu_\infty^{\natural, (N)},\la_\infty^{\natural,
 (N)}}(q).
\end{eqnarray*}
Similarly, we have
$
 \ell_{\mu,\la}(q)
 = \mathfrak l_{\mu_\infty^{\natural,
 (N)},\la_\infty^{\natural,(N)}}(q).$

The general case of $u_{\mu,\la}(q), \ell_{\mu,\la}(q)$ for $\la,
\mu \in X_{\m|n}^{+}$ can be easily reduced to the case considered
above as follows. Let ${\bf 1}_{m|n}:=(\overbrace{-1,\ldots,
-1}^m|\overbrace{1,\ldots,1}^n) \in X_{\m|n}^{++}$. Note that
$u_{\mu,\la}(q) =u_{\mu+k{\bf 1}_{m|n},\la+k{\bf 1}_{m|n}}(q),
\ell_{\mu,\la}(q) =\ell_{\mu+k{\bf 1}_{m|n},\la+k{\bf
1}_{m|n}}(q)$, and also that $\la+k{\bf 1}_{m|n} \in
X_{\m|n}^{++}$, for $\la\in X_{\m|n}^{+}$ and $k \gg 0$.

Thus our result follows from the corresponding well-known
positivity results of Kazhdan-Lusztig polynomials which was proved
using deep geometric techniques \cite{KL2, BB, BK}.

(2) Let $1_{m|n}\in\Z^{{\bf m}|n}$ denote function given by
$1_{m|n}(i)=1$, for all $i\in I(m|n)$. The formula for $U_{f-k{
1}_{m|n}}$, with $k \in\Z$, is obtained from $U_f$ by shifting the
weights in the monomials that appear in $U_f$ by $-k{1}_{m|n}$. Also
if we write $X_a^{(r)} U_f=\sum_{g}x_{gf}(q) U_g$, with
$x_{gf}(q)\in\Z[q,q^{-1}]$, then $X_{a-k}^{(r)}
U_{f-k1_{m|n}}=\sum_{g}x_{gf}(q) U_{g-k 1_{m|n}}$ (here $X=E,F$).
Thus it suffices to verify (2) within ${{\mathcal E}}^{{\bf m}|n}_+$
by assuming $a<n$ and $f \in\Z_{++}^{\m|n}$. Using the truncation
maps we can pass to the case when $n=\infty$ (see
Corollary~\ref{aux43}). By Proposition~\ref{Fockiso} and
Theorem~\ref{correspondence}, this amounts to prove the
corresponding statement for $\mathcal U_f$ and $\mathcal L_f$ in
$\widehat{\mathcal E}^{\m+\infty}$. But this follows from the
validity of the corresponding statement in $\widehat{\mathcal
E}^{\m+n}$ for $n$ finite (see Remark~\ref{rem:folklore}) and the
property of the truncation map $\textsf{Tr}_{\infty,n}$ in
Proposition~\ref{commutativity-reduct}.
\end{proof}

As explained in \cite[2-k]{Br}, the positivity in
Theorem~\ref{th:BrConj}~(2) together with (a parabolic variant of)
the algorithm in \cite[2-j]{Br} for computing the canonical basis
elements in $\widehat{\mathcal E}^{\m|n}$ imply the following.

\begin{corollary}
Let $n$ be finite. Every canonical basis element $U_f$ in the
completion $\widehat{\mathcal E}^{\m|n}_+$ actually lies in
${\mathcal E}^{\m|n}_+$, that is, $U_f$ is a finite sum of
monomials $K_g$.
\end{corollary}

\begin{remark}
Such a finiteness of canonical basis elements in $\mathcal
E^{\m|n}$ supports Conjecture~\ref{conj:BrKL}, since it is
compatible with the fact that a Verma flag of any tilting module
in $\FPn$ is finite.
\end{remark}

\begin{corollary} \label{cor:finite}
Let $\infty \ge n >n_0$, $f \in \Z^{\m|n_0}_{++}$, and extend $f$
to $f^{(n)} \in \Z^{\m|n}_{++}$ by letting $f^{(n)} (i)=i$ for
$n_0<i\le n$.  Let $n_f \gg 0$ be the smallest integer such that
$\# f^{(n)} =\# f$, for all $n \ge n_f$. Then $U_{f^{(n)}}$
contains the same (finite) number of monomials for all $\infty \ge
n \ge n_f$.
\end{corollary}

\begin{proof}
Let $\infty \ge n \ge n_f$. Also write $U_{f^{(n)}} =\sum_{g
\preceq f^{(n)}} u_{g,f^{(n)}}(q) K_g$. It follows from $g \preceq
f^{(n)}$ and $n \ge n_f$ that $g \in X_{\m|n}^{++}$ and
$g=g_1^{(n)}$ for $g_1 \in X_{\m|n_f}^{++}$. Recall that
$\mathfrak{Tr}_{n,n_f} (K_{g}) =K_{g_1}, \mathfrak{Tr}_{n,n_f}
(U_{f^{(n)}}) =U_{f^{(n_f)}}$. Thus when applying the truncation
map $\mathfrak{Tr}_{n,n_f}$ to the previous identity for
$U_{f^{(n)}}$, every nonzero monomial survives, and we obtain that
$U_{f^{(n_f)}} =\sum_{g_1 \preceq f} u_{g,f^{(n)}}(q) K_{g_1}$.
\end{proof}

\begin{corollary} \label{cor:finite-reduct}
Let $n >n_0$, $f \in \Z^{\m +n_0}_{++}$, and extend $f$ to
$f^{(n)} \in \Z^{\m+n}_{++}$ by letting $f^{(n)} (i)=1-i$ for
$n_0<i\le n$. Then, there exists $n_f \gg 0$ such that the number
of monomials in $\mathcal U_{f^{(n)}}$ is independent of $n \ge
n_f$.
\end{corollary}

\begin{proof}
By a truncation map argument similar to the proof of
Corollary~\ref{cor:finite}, the number of monomial terms in
$\mathcal U_{f^{(n)}}$ is weakly increasing as $n$ increases. But
this number has to stabilize, since it is bounded according to
Corollary \ref{cor:finite} and Theorem \ref{correspondence} (4).
\end{proof}

\subsection{The category $\OPn$}
\label{sec:categoryO}

Let $n \in \N $. We shall think of ${\mathfrak g\mathfrak l}(m+n)$
as the Lie algebra of complex matrices whose rows and columns are
parameterized by $I(m|n)$. Let $e_{ij}$, $i,j\in I(m|n)$ be the
elementary matrices. We denote by $\mathfrak h_c$ (respectively
$\mathfrak b_c$) the standard Cartan (respectively Borel)
subalgebra of ${\mathfrak g\mathfrak l}(m+n)$, which consists of
the diagonal (respectively the upper triangular) matrices. Let
$\{\delta'_i, i \in I(m|n)\}$ be the basis of $\mathfrak h_c^*$
dual to $\{e_{ii}, i \in I(m|n)\}$. Introduce the Levi subalgebra
$\mathfrak l = \oplus_{i=1}^s {\mathfrak g\mathfrak l}(m_i)\oplus
{\mathfrak g\mathfrak l}(n)$ and the corresponding parabolic
subalgebra $\mathfrak q =\mathfrak l +\mathfrak b_c$ of $\mathfrak
g\mathfrak l(m+n)$. Let $\mathfrak{gl}(m+\infty)
=\lim\limits_{\stackrel{\longrightarrow}{n}}\mathfrak{gl}(m+n)$.

Define the symmetric bilinear form $(\cdot\vert\cdot)_c$ on
$\mathfrak h_c^*$ by
$$(\delta'_i\vert\delta'_j)_c=\delta_{ij}, \qquad i,j \in I(m|n).$$
Let $X_{m+n}$ be the set of integral weights $\la=\sum_{i \in
I(m|n)} \la_i \delta'_i$, $\la_i \in\Z$. Define
\begin{align*}
X^+_{\m+n} &:= \{\la \in X_{m+n} \mid
  \la_{-m} \ge \cdots \ge\la_{-m+m_1-1}, \\
& \qquad\quad \la_{-m+m_1} \ge \cdots \ge \la_{-m+m_1+m_2-1},\\
& \qquad\quad \cdots, \la_{-m_s} \ge \cdots \ge \la_{-1}, \la_1
\ge \cdots \ge \la_n\},   \\
X_{\m+n}^{++} &:= \{\la \in X^+_{\m+n} \mid  \la_n \geq 0 \}.
\end{align*}

We may regard an element $\la$ in $X^{++}_{\m+n}$ as an element in
$X^{++}_{\m+n'}$ for $n'>n$ by adjoining zeros. Set
$$X^{++}_{\m+\infty} \equiv X^{+}_{\m+\infty} :=\lim\limits_{\stackrel{\longrightarrow}{n}} X^{++}_{\m+n}.
$$
For $n \in \N \cup \infty$ define
$$\rho' =-\sum_{i=-m}^{-1}i\delta'_i+\sum_{j=1}^n(1-j)\delta'_j.$$
Define a bijection
\begin{eqnarray}
 X_{\m+n} \longrightarrow \Z^{\m+n}, \qquad \la \mapsto f_\la,
\end{eqnarray}
where $f_\la \in \Z^{\m+n}$ is given by $f_\la (i) = (\la +\rho' |
\delta_i')_c$ for all $i \in I(m|n)$.
This map induces bijections $X^{+}_{\m+n} \rightarrow \Zmpn$, and
$X^{++}_{\m+n} \rightarrow \Z_{++}^{\m+n}$. Using this bijection
we define the notions such as $\epsilon$-weight, partial order
$\leq$, et cetera, for elements in $X^+_{\m+n}$ by requiring them
to be compatible with those defined for elements in $\Zmpn$.

Given $\la \in X_{m+n}^+$, $n\in\N\cup\infty$, we define as usual
the parabolic Verma module
$${\mathcal K}_n(\la)
:=U({\mathfrak g\mathfrak l}(m+n)) \otimes_{U(\mathfrak q)}
L_n^0(\la)$$
and its irreducible quotient ${\mathfrak g\mathfrak
l}(m+n)$-module $\mathcal L_n(\la)$.

Let $n\in\N$. Denote by $\OPn$ the category of finitely generated
$\glmpn$-modules $M$ that are locally finite over $\mathfrak q$,
semisimple over $\mathfrak l$ and
\begin{equation*}
M=\bigoplus_{\g\in X_{m+n}}M_\g,
\end{equation*}
where as usual $M_\g$ denotes the $\g$-weight space of $M$ with
respect to $\mathfrak h_c$. The parabolic Verma module $\mathcal
K_n(\la)$ and the irreducible module $\mathcal L_n(\la)$ for $\la
\in X_{m+n}^+$ belong to $\OPn$. Denote by $\OPPn$ the full
subcategory of $\OPn$ which consists of $\glmpn$-modules $M$ whose
composition factors are isomorphic to $\mathcal L_n(\la)$ with
$\la \in X^{++}_{\m+n}$. Given $M\in\OPn$, we endow the restricted
dual $M^*$ with the usual $\mathfrak g\mathfrak l(m+n)$-module
structure. Further twisting the ${\mathfrak g\mathfrak
l}(m+n)$-action on $M^*$ by the automorphism given by the negative
transpose of $\mathfrak g\mathfrak l(m+n)$, we obtain another
$\mathfrak g$-module denoted by $M^\tau$.

Tilting modules $\mathcal U_n (\la)$ for $\la \in X_{\m+n}^+$ in
$\OPn$ were constructed as in \cite{CoI, So2} and are known to have
Verma flags. The character formula of the tilting module $\mathcal
U_n (\mu)$ in $\OPn$ is given by \cite{So2}: for $\la, \mu \in
X_{\m+n}^+$,
\begin{eqnarray} \label{eq:tiltU}
(\mathcal U_n (\la): \mathcal  K_n(\mu))
 =[\mathcal K_n(-w_0\mu- 2\rho' +2\rho_{\mathfrak l}):
\mathcal L_n(-w_0\la- 2\rho' +2\rho_{\mathfrak l})].
\end{eqnarray}

We remark that for $n\in \N \cup \infty$ the $\mathfrak g\mathfrak
l (m+n)$-module ${\mathcal K}_n(\la)$ is irreducible if and only
if $\la$ is a minimal weight in $X^+_{\m+n}$ in the Bruhat
ordering.

Denote by $\Oi$ the category  of finitely generated
$\mathfrak{gl}(m+\infty)$-modules that are $\mathfrak l$-semisimple,
locally finite over $\mathfrak q\cap \mathfrak{gl}(m+N)$, for every
$N$, and such that the composition factors are of the form $\mathcal
L(\la)$, $\la\in X_{\m+\infty}^+$.

\subsection{Kazhdan-Lusztig theory and (dual) canonical bases}

We will write $\mathfrak l_{g,f}(q), \mathfrak t_{g,f}(q)$ for
$\mathfrak l_{\mu,\la}(q), \mathfrak t_{\mu,\la}(q)$, where $f$,
$g$ correspond to $\la, \mu$, respectively, under the bijection
$X_{m+n}^{+} \rightarrow \Z_+^{m+n}$.

The following is a increasingly better known reformulation, in
terms of dual canonical and canonical bases, of the
Kazhdan-Lusztig conjecture, proved in \cite{BB, BK}, combined with
the translation principle and the character formula of tilting
modules \cite{So2}. The proof in \cite[Theorem~5.4]{CWZ} for the
special case (i.e. $\m =m$) works in the current setup as well
(also cf. Brundan-Kleshchev \cite{BKl}).

\begin{theorem} \label{th:multiCan}
In the Grothendieck group $G(\OPn)$, for $\nu \in X_{m+n}^{+}$, we
have
$$
[\mathcal U_n(\nu)] =\sum_{\mu \in X_{\m+n}^{+}} \mathfrak
u_{\mu,\nu}(1) [\mathcal K_n(\mu)].
$$
\end{theorem}

Theorem~\ref{th:multiCan} is equivalent to the following character
formula by Remark~\ref{dualUK} and (\ref{eq:tiltU}):
$$
\text{ch} \mathcal L_n(\nu) =\sum_{\mu \in X_{\m+n}^{+}} \mathfrak
l_{\mu,\nu}(1) \text{ch}\, \mathcal K_n(\mu).
$$

Recall the $\epsilon$-weight on $X_{m+n}$ defined in
(\ref{nonsuperweight}). Denote by $\chi_\la$ the central character
associated to $\la \in X_{m+n}$. By Harish-Chandra's theorem
$\chi_\la =\chi_\mu$ for $\la, \mu \in X_{m+n}$ if and only if
$\la =\sigma \cdot \mu$ for some $\sigma \in S_{m+n}$, or
equivalently $\text{wt}^\epsilon(\la)=\text{wt}^\epsilon(\mu)\in
P$. We denote by $\mathcal O^+_\g$ the block in $\OPn$ associated
to $\g \in P$. Let $V$ be the natural $\mathfrak g\mathfrak l
(m+n)$-module and $V^*$ its dual. For $a\in \Z, r\geq 1$ we define
the translation functors $E_a^{(r)}, F_a^{(r)}: \OPn
\longrightarrow \OPn$ by sending $M \in \mathcal O^+_\g$ to
\begin{eqnarray*}
 F_a^{(r)} M := \text{pr}_{\g - r(\epsilon_a
 -\epsilon_{a+1})} (S^rV \otimes M),\quad
 E_a^{(r)} M := \text{pr}_{\g + r(\epsilon_a
 -\epsilon_{a+1})} (S^rV^* \otimes M).
\end{eqnarray*}
Let $\OnDelta$ be the full subcategory of $\OPn$ consisting of all
modules with Verma flags. Let $G(\OnDelta)_{\mathbb
Q}:=G(\OnDelta)\otimes_\Z{\mathbb Q}$ and let $\mathcal
E^{\m+n}|_{q=1}$ be the specialization of $\mathcal E^{\m+n}$ as
$q\to 1$.

\begin{theorem}\label{red-KL-finite}
Let $n \in \N$.
\begin{enumerate}
\item Sending the Chevalley generators $E_a^{(r)}, F_a^{(r)}$
$(a\in\Z, r\ge 1)$ to the translation functors $E_a^{(r)},
F_a^{(r)}$ defines a $\mathcal U_{q=1}$-module structure on
$G(\OnDelta)_{\mathbb Q}$.

\item The linear map $i: G(\OnDelta)_{\mathbb Q} \rightarrow
\mathcal E^{\m+n}|_{q=1}$, which sends $[\mathcal K_n(\la)]$ to
$\mathcal K_{f_\la}(1)$, for each $\la \in X^+_{\m+n}$, is an
isomorphism of $\mathcal U_{q=1}$-modules.

\item The map $i$ sends $[\mathcal U_n(\la)]$ to $\mathcal
U_{f_\la}(1)$, for each $\la \in X^+_{\m+n}$.
\end{enumerate}
\end{theorem}
Equivalently, for $\la, \mu \in X^+_{\m+n}$, we have
\begin{eqnarray*}
(\mathcal U_n(\la): \mathcal K_n(\mu)) &=& \mathfrak u_{\mu,\la}(1)  \\
{[}\mathcal K_n(\la): \mathcal L_n(\mu)] &=& \mathfrak u_{-w_0\la-
2\rho'+2\rho_{\mathfrak l}, -w_0\mu-2\rho'+2\rho_{\mathfrak
l}}(1).
\end{eqnarray*}

\begin{proof} The map $i$ is certainly a vector space isomorphism.
One checks that the action of the translation functors on the
parabolic Verma modules is compatible with the action of the
divided powers of the Chevalley generators of $\mathcal U_{q=1}$
on the monomial basis. Thus (1) and (2) follow.  Now (3) follows
from Theorem \ref{th:multiCan} and the definition of KL
polynomials $\mathfrak u_{\mu,\nu}$ and $\mathfrak l_{\mu,\nu}$.
\end{proof}

\subsection{The case as $n \mapsto \infty$}
By studying truncation functors ${\textsf{Tr}}$ for $\OPn$ with
varying $n$, analogous to Subsection~\ref{subsec:truncat} (cf.
\cite{Don}), we can establish the counterparts of
Subsection~\ref{subsec:inftytilt}.

The following theorem should be compared to Theorem~\ref{tiltInf}.
Note that Corollary~\ref{cor:finite-reduct} is used in proving (4)
below.
\begin{theorem}\label{tiltInf:red}
Let $\la\in X^{+}_{\m+\infty}$.
\begin{enumerate}
\item There exists a unique tilting module $\mathcal U(\la)$ in
$\Oi$ with $\mathcal K(\la)$ sitting at the bottom of a Verma
flag. Moreover, $\mathcal U(\la) = \cup_n \mathcal U_n(\la)$.

\item The functor ${\textsf{tr}}_{n}$ sends $\mathcal U (\la)$ to
 $\mathcal U_n(\la)$ if $(\la|\delta_{n+1})_c=0$ and to $0$ otherwise.

\item We have $(\mathcal U(\la):\mathcal K(\mu))=(\mathcal
U_n(\la):\mathcal K_n(\mu))$ for $n\gg0$.

\item The Verma flag lengths for $\mathcal U(\la)$ and $\mathcal
U_n(\la)$ for $n\gg 0$ are the same (and finite).
\end{enumerate}
\end{theorem}

The following proposition follows from Theorem~\ref{th:multiCan},
Corollary~\ref{cor:finite-reduct}, and the properties of the
truncation maps/functors.

\begin{proposition}
Let $n >n_0$ and $\la \in X_{\m +n_0}^{++}$. Extend $\la$ to
$\la^{(n)} \in X_{\m+n}^{++}$ by letting $\la^{(n)} (i)=0$ for
$n_0<i\le n$. Then, there exists $n_\la \gg 0$ such that the Verma
flag structure of $\mathcal U_{\la^{(n)}}$ is independent of $n
\ge n_\la$.
\end{proposition}

\subsection{A general super duality conjecture}\label{sec:isom}

Based on Conjecture~\ref{conj:BrKL}, Theorems~\ref{red-KL-finite}
and \ref{correspondence} we propose the following conjecture which
generalizes \cite[Conjecture~6.10]{CWZ}, which will be referred to
as the {\em general super duality conjecture}.

\begin{conjecture} \label{conj:duality}
For a tuple of positive integers $\m$, the categories $\Fi$ and
$\Oi$ are equivalent.
\end{conjecture}

\begin{remark}
We regard Conjecture~\ref{conj:duality} as a pointer toward a
profound connection between representation theories of Lie algebras
and Lie superalgebras. One should keep in mind some variations of
the conjecture such as an isomorphism of the full subcategories of
modules with Verma flags, or an equivalence of derived categories,
et cetera.

The validity of Conjecture~\ref{conj:duality} implies the validity
of the parabolic Brundan-Kazhdan-Lusztig Conjecture~\ref{conj:BrKL},
by using the Kazhdan-Lusztig theory for $\mathfrak{gl}(m+n)$ as
formulated in Theorem~\ref{red-KL-finite} and the properties of
truncation maps/functors (see Corollary~\ref{aux43} and
Proposition~\ref{nestedtilt3}). In particular, the original Brundan
conjecture for the full category $\mathcal O$ of $\gl(m|1)$-modules
(cf. Remark~\ref{rem:genO}) would follow from the super duality
conjecture.
\end{remark}

\section{Application of the Chuang-Rouquier $\mathfrak{sl}_2$-categorification}
\label{sec:categorification}

\subsection{The $\mathfrak{sl}_2$-categorification and category $\FPn$}

This subsection is a super analogue of Chuang-Rouquier
\cite[7.4]{CR}.

Let $\{u_i\}$ be a $\Z_2$-homogeneous basis of $\mathfrak g
=\glmn$, and $\{u^i\}$ be its dual basis with respect to the
supersymmetric bilinear form $\langle a,b \rangle := \text{str}
(ab)$, where $ab$ denotes the matrix multiplication of $a,b \in
\glmn$. The Casimir $C:= \sum_i (-1)^{|u_i|} u_i u^i$ lies in the
center of the enveloping algebra $U(\mathfrak g)$. By means of the
standard matrix elements, we readily see that
$$C = \sum_{i,j\in I(m|n)} (-1)^{\bar{j}} e_{ij} e_{ji}.
$$
Recall that $\bar{j} =0$ if $j<0$ and $\bar{j}=1$ if $j>0$. Denote
by $\{x_i\}_{i\in I(m|n)}$ the standard basis for the natural
$\mathfrak g$-module $V$, and set $|x_i|=\bar{i}$.

Given a $\mathfrak g$-module $M$, we let $X_M \in
\text{End}_{\mathfrak g} (V \otimes M)$ the adjoint map associated
to the action map $\mathfrak g \times M \rightarrow M$ (by
identifying $\mathfrak g =\text{End} (V)$). It follows that
$$ X( v \otimes m) = \Omega (v \otimes m),$$
where
$$ \Omega =\sum_{i,j\in I(m|n)} (-1)^{\bar{j}} e_{ij} \otimes e_{ji}.$$
This defines an endomorphism $X$ of the functor $V\otimes -$. One
verifies that (with all the superalgebra signs cancelling)
\begin{eqnarray} \label{eq:Casimir}
\Omega =\hf (\Delta_{\mathfrak g} (C) - C\otimes 1 - 1 \otimes C),
\end{eqnarray}
where $\Delta_{\mathfrak g}$ denotes the coproduct on $U(\mathfrak
g)$. We also define
$$T_M \in \text{End}_{\mathfrak g} (V \otimes V
\otimes M),
 \qquad v \otimes v' \otimes m \mapsto (-1)^{|v||v'|}
v'\otimes v \otimes m.
$$
This defines an endomorphism $T$ of the functor $V\otimes V
\otimes -$.

Recall that the {\em degenerate affine Hecke algebra} $H_\ell$ is
an algebra generated by $X_i (i =1,\ldots, \ell)$ and $s_i (i
=1,\ldots, \ell-1)$, subject to the following relations:
\begin{eqnarray*}
s_i^2 =1,\quad s_i s_{i+1} s_i &=& s_{i+1} s_i s_{i+1},  \nonumber\\
s_i s_j &=& s_j s_i,\quad |i-j|>1,    \nonumber \\
x_j s_i &=& s_i x_j, \quad (j \neq i,i+1),  \nonumber \\
x_{i+1} s_i -s_i x_i &=& 1, \label{hecke} \\
x_i x_j &=& x_j x_i, \quad (i \neq j). \nonumber
\end{eqnarray*}

The following is a super generalization of a theorem of
Arakawa-Suzuki \cite{AS}.

\begin{proposition} \label{Heckemodule}
There is an algebra homomorphism
\begin{align*}
H_\ell  & \longrightarrow \text{End}_{\mathfrak g} (V^{\otimes
\ell}
\otimes M), \\
s_i & \mapsto {\bf 1}_V^{\otimes \ell -i} \otimes T_{V^{\otimes
{i-1}} \otimes M}, \qquad
x_i \mapsto {\bf 1}_V^{\otimes \ell-i} \otimes X_{V^{\otimes
i-1}\otimes M}.
\end{align*}
\end{proposition}
\begin{proof}
All the relations are straightforward to check except
(\ref{hecke}). The relation (\ref{hecke}) is equivalent to the
following identity in $\text{End}_{\mathfrak g} (V \otimes V
\otimes M)$ for $\mathfrak g$-module $M$:
$$ T_M \circ ({\bf 1}_V \otimes X_M) = X_{V\otimes M} \circ T_M -
{\bf 1}_{V\otimes V \otimes M}.
$$

Indeed, given $a, b \in I(m|n)$, we calculate that
\begin{align*}
X_{V\otimes M}T_M (x_a \otimes & x_b \otimes m)
= (-1)^{\bar{a}\bar{b}} X_{V\otimes M} (x_b \otimes x_a \otimes m) \\
&= \sum_{i,j \in I(m|n)}
(-1)^{\bar{a}\bar{b}+\bar{j}+(\bar{i}+\bar{j})\bar{b}} e_{ij} x_b
\otimes e_{ji}
(x_a \otimes m) \\
&= \sum_{i,j \in I(m|n)}
(-1)^{\bar{a}\bar{b}+\bar{j}+(\bar{i}+\bar{j})\bar{b}} e_{ij} x_b
\otimes e_{ji}
x_a \otimes m \\
& + \sum_{i,j \in I(m|n)}
(-1)^{\bar{a}\bar{b}+\bar{j}+(\bar{i}+\bar{j})(\bar{b}+\bar{a})}
e_{ij} x_b \otimes x_a \otimes e_{ji} m
\\
&= x_a \otimes x_b \otimes m + T_M \sum_{i,j \in I(m|n)}
(-1)^{\bar{j}+(\bar{i}+\bar{j})\bar{b}}
x_a \otimes e_{ij} x_b \otimes e_{ji} m \\
&=\big ({\bf 1}_{V\otimes V \otimes M} +T_M \circ ({\bf 1}_V
\otimes X_M) \big) (x_a \otimes x_b \otimes m).
\end{align*}
\end{proof}

We write $\la \rightarrow_a \mu$ if there exists $i \in I(m|0)$
such that $\la_i -i =a, \mu_i-i =a+1$, or if there exists $i \in
I(0|n)$ such that $-\la_i+i =a+1, -\mu_i+i=a$, and in addition,
$\la_j =\mu_j$, for all $j \neq i$. Given two (integral) blocks
$\mathcal O^+_\g, \mathcal O^+_{\g'}$ in the category $\FPn$
corresponding to $\g, \g' \in P$, we write $\g\rightarrow_a \g'$
if there exists $\la, \mu \in \mathfrak h^*$ such that $K_n(\la)
\in \mathcal O^+_\g$ and $K_n(\mu) \in \mathcal O^+_{\g'}$. Denote
by $\pr_\g$ the projection onto the block $\mathcal O^+_\g$. We
can rewrite the translation functors $F_a$ (\ref{eq:transl}) as
$$
F_a = \bigoplus_{\g,\g': \g\rightarrow_a \g'} \pr_{\g'} \circ (V
\otimes -) \circ \pr_\g.
$$

\begin{proposition} \label{trans=casim}
The translation functor $F_a$ can be identified with the
generalized $(a-m)$-eigenspace of $X$ acting on $V \otimes -$.
\end{proposition}

\begin{proof}
It suffices to check the proposition on a parabolic Verma module
$K_n(\la)$. The Casimir acts on $K_n(\la)$ as the scalar
multiplication by
$c_\la := \langle \la+2 \rho, \la \rangle.$ By (\ref{eq:Casimir}),
$\Omega$ acts on a subquotient $K_n(\la +\delta_i)$ in $V\otimes
K_n(\la)$ (where we recall $V =L_n(\delta_{-m})$) as the
multiplication by
\begin{align*}
\hf & (c_{\la +\delta_i} -c_\la -c_{\delta_{-m}}) \\
& = \hf ( \langle \la +\delta_i+2 \rho, \la +\delta_i \rangle -
\langle \la+2 \rho, \la \rangle
- \langle \delta_{-m}+2 \rho, \delta_{-m} \rangle ) \\
% &=
% 2 \langle \la, \delta_i\rangle +
% \langle \delta_i +2\rho, \delta_i\rangle
% -\langle \delta_1 +2\rho, \delta_1 \rangle \\
&= \langle \la, \delta_i\rangle + \hf \langle \delta_i,
\delta_i\rangle -\hf \langle \delta_{-m}, \delta_{-m} \rangle
-\langle \rho, \delta_{-m} -\delta_i \rangle \\
&= \left\{
\begin{array}{ll}
\la_i -i-m, & \quad \text{if } i \in I(m|0) \\
-\la_i +i-m-1, & \quad \text{if } i \in I(0|n).
\end{array}
\right.
\end{align*}
The statement now follows by comparing with the definition of
$F_a$.
\end{proof}

We can identify $E_a$ similarly. Note that the notations $E$ and
$F$ are switched in \cite{CR}. Following \cite[7.4]{CR},
Propositions~\ref{Heckemodule} and \ref{trans=casim} above imply
that $E_a,F_a,X,T$ satisfy the definition of the
$\mathfrak{sl}_2$-categorification (which we will skip here and
refer to \cite[5.1.1, 5.2.1]{CR} for detail).

\subsection{A formal consequence}
%analogue of Chuang-Rouquier \cite[5.4]{CR}}

By definition, the (divided power) translation functors $E_a^{(i)},
F_a^{(i)}$ for $i \ge 1$, are obtained from the functors $E^i, F^i$
by replacing $V^{\otimes i}$ by the symmetric products $S^i V^*,
S^iV$ et cetera).

We shall need the following formal consequence of the
$\mathfrak{sl}_2$-categorification (see
\cite[Proposition~5.23]{CR} and a statement in its proof).

\begin{theorem}\label{cr:simplicity}
For every simple object $L$ in $\FPn$, $i \le d:=\max \{j| F^j(L)
\neq 0\}$, and $a\in\Z$, the socle and cosocle of $F^{(i)}L$ are
simple and isomorphic. Furthermore, $F^{(d)} L$ is simple.
\end{theorem}

\section{Some results on canonical basis and tilting modules}
\label{sec:partial}

In this section we establish some miscellaneous results on
canonical basis elements and tilting modules that will be used in
subsequent sections.

\subsection{The $\texttt L$ operators}

Let $n\in\N$ and $f \in \Z^{m|n}$ be $S_{\m|n}$-conjugate to an
element in $\Z^{\m|n}_+$. Recall that $m=m_1+\cdots+m_s$. Let $-m
\le i <0 <j \le n$ with $f(i) =f(j)$. We define the $\texttt L$
operators (cf.~\cite{Br})
\begin{eqnarray*}
\texttt L_{i,j} (f) := f -a(d_i -d_j),
\end{eqnarray*}
where $a$ is the smallest positive integer such that $f -a(d_i
-d_j)$ and all $\texttt L_{k,l}(f) - a(d_i -d_j)$ for $-m\le
i,k<0<l<j\le n$ with $f(k) =f(l)$ are $S_{\m|n}$-conjugate to
elements of $\Z_{+}^{\m|n}$.

Now let $f\in\Z^{\m|n}_+$ and suppose $\#f=k$. Let $-m\le
i_i,i_2,\ldots,i_k\le -1$ and $1\le j_k<j_{k-1}<\cdots<j_1\le n$
be such that $f(i_l)=f(j_l)$, for $l=1,\ldots,k$.  For a $k$-tuple
$\theta=(\theta_1,\ldots,\theta_{k})\in\N^{k}$ we define \cite{Br}
\begin{align*}
f^{\texttt{L}_\theta} \equiv {\texttt L}_\theta(f)
=\Big{(}{\texttt L}^{\theta_k}_{i_k, j_k} \circ\cdots
 \circ {\texttt L}^{\theta_1}_{i_1 ,j_1}(f)\Big{)}^+,
 %\quad
%{\texttt L}'_\theta(f) =\Big{(}{\texttt
%L}^{\theta_1}_{i_1,j_1}\circ\cdots
 %\circ {\texttt L}^{\theta_k}_{i_k ,j_k}(f)\Big{)}^+,
\end{align*}
where the superscript $+$ here stands for the unique
$S_{\m|n}$-conjugate in $\Z^{\m|n}_+$.
\subsection{The positive pairs}

In this subsection we set $\m=(m_1,m_2)$ with $m_1+m_2 =m$, and
shall adapt here the notion of positive pairs defined in
$\Z_+^{m_1,m_2|0}$ from \cite{CWZ}.

Let $f \in \Z_+^{m_1,m_2|n}$. For a pair of integers $(i|j)$ such
that $-m\le i < -m_2 \le j<0$, we define the {\em distance} of
$(i|j)$ (associated to $f\in\Z^{\m+n}_+$) to be
$d(i|j):=f(i)-f(j).$
We call $(i|j)$ an {\em admissible pair} for $f$ if $f(i)>f(j)$
and $f\cdot \tau_{ij}$ affords a (unique) conjugate $(f\cdot
\tau_{ij})^+ \in\Z^{\m|n}_+$.
Two admissible pairs $(i_1|j_1)$ and $(i_2 |j_2)$ for $f$ are said
to be {\em disjoint}, if $i_1\not=i_2$ and $j_1\not=j_2$. Two
subsets $A_1$ and $A_2$ of admissible pairs of $f$ are said to be
disjoint, if any two admissible pairs $(i_1|j_1)\in A_1$ and $(i_2
|j_2)\in A_2$ are disjoint.
Let $A^+_f$ denote the set of all admissible pairs of $f$.  For $k
\ge 1$ we define recursively
%\begin{align*}
$\Sigma^k_f:=\{(i|j)\in A^+_f\vert d(i|j)=k\text{ and }
(i|j)\text{ disjoint from } \sqcup_{s=1}^{k-1}\Sigma^s_f\}. $
%\end{align*}
Let $\Sigma^+(f)\equiv \Sigma^+_f:=\sqcup_{k\ge 1} \Sigma^k_f$. An
element in $\Sigma^+_f$ is called a {\em positive pair} of $f$.
Given a subset $\Sigma$ of positive pairs of $f$, we denote by
$f_\Sigma$ the element in $\Z^{\m|n}_+$ obtained by first
interchanging the values of $f$ at each positive pair in $\Sigma$,
and then taking the unique $S_{\m|n}$-conjugate in $\Z^{\m|n}_+$.

Set $I(m|n) =I_1 \sqcup I_2 \sqcup I_3$, where $I_1, I_2, I_3$ are
the increasing subintervals of $I(m|n)$ of length $m_1,m_2,n$
respectively. We denote by $f_{ab}$ the restriction of $f$ to
$I_a\cup I_b$ with $a\leq b$ and let $f_a =f_{aa}$.

\subsection{On tilting modules in $\mathcal O^+_{m_1,m_2|n}$}
Given $\la\in X_{\m|n}^+$, by abuse of notation we denote
$U_n(f_\la)=U_n(\la)$, $K_n(f_\la)=K_n(\la)$ and
$L_n(f_\la)=L_n(\la)$, the respective tilting, parabolic Verma and
irreducible $\gl(m|n)$-modules.

Let $f\in\Z^{m_1,m_2|n}_+$ with $m=m_1+m_2$. We denote by
$K^{12}(f_{12})$ the parabolic Verma $\gl(m_1+m_2)$-module.
Likewise the notation $K^{23}(f_{23})$ denotes the parabolic Verma
$\gl(m_2|n)$-module.

\begin{proposition}\label{condition:can}
Let $f\in\Z^{m_1,m_2|n}_+$.
\begin{itemize}
\item[(i)] Let $\Sigma$ be a positive pair of $f$. Assume that
${\rm Ext}^1(K^{12}((f_{\Sigma})_{12}),K^{12}(f_{12}))\not=0$.
Suppose that for any $g\prec f$ with $(U_n(f)):K_n(g))\not=0$ we
have $g\not\succ f_{\Sigma}$. Then $(U_n(f):K_n(f_\Sigma))\ge 1$.

\item[(ii)] Let $-m_2\le i<0<j\le n$ be such that $f(i)=f(j)$.
Assume that ${\rm Ext}^1(K^{23}(f^{{\texttt
L}_{i,j}}_{23}),K^{23}(f_{23}))\not=0$. Suppose that for any
$g\prec f$ with $(U_n(f)):K_n(g))\not=0$ we have $g\not\succ
f^{{\texttt L}_{i,j}}_{23}$. Then $(U_n(f):K_n(f^{{\texttt
L}_{i,j}}_{23}))\ge 1$.
\end{itemize}
\end{proposition}

\begin{proof}
We will only show (i).  The proof of (ii) is similar.

Let $\Sigma=(i|j)$ be such that $-m \le i<-m_2\le j<0$ and
$f(i)>f(j)$. Let $g=f_\Sigma$. Let $L_0^{12}(f_{12})$ and
$L_0^{12}(g_{12})$ denote the irreducible $\mathfrak{p}_{m_1,m_2}$
modules of highest weights $f_{12}$ and $g_{12}$ respectively,
where $\mathfrak{p}_{m_1,m_2}$ is the parabolic subalgebra of
$\gl(m_1+m_2)$ with Levi subalgebra $\gl(m_1)\oplus\gl(m_2)$. By
assumption there exists a non-split extension $T$ of
$\gl(m_1+m_2)$-modules
\begin{equation*}\label{ext111}
0\longrightarrow{\rm
Ind}_{\mathfrak{p}_{m_1,m_2}}^{\gl(m_1+m_2)}L_0^{12}(f_{12})\longrightarrow
T\longrightarrow {\rm
Ind}_{\mathfrak{p}_{m_1,m_2}}^{\gl(m_1+m_2)}L_0^{12}(g_{12})\longrightarrow
0.
\end{equation*}
Tensoring the above sequence with the simple $\gl(n)$-module
$L^3(f_3)=L^3(g_3)$, we obtain a non-split extension of
$\gl(m_1+m_2)\oplus\gl(n)$-modules
\begin{equation}\label{ext123}
0\rightarrow{\rm
Ind}_{\mathfrak{p}_{m_1,m_2}}^{\gl(m_1+m_2)}L_0^{12}(f_{12})\otimes
L^3(f_3)\rightarrow T'\rightarrow {\rm
Ind}_{\mathfrak{p}_{m_1,m_2}}^{\gl(m_1+m_2)}L_0^{12}(g_{12})\otimes
L^3(g_3)\rightarrow 0.
\end{equation}

By applying an induction functor to (\ref{ext123}), we obtain a
short exact sequence of $\gl(m_1+m_2|n)$-modules
\begin{equation}\label{ext456}
0\rightarrow K_n(f)\rightarrow T''\rightarrow K_n(g)\rightarrow 0.
\end{equation}
Taking the invariants of (\ref{ext456}) with respect to the
niradical of the parabolic $\mathfrak{p}$ whose Levi is
$\gl(m)\oplus\gl(n)$ we recover (\ref{ext123}), and hence the
indecomposability of $T''$ follows from that of $T'$.

Finally from the construction of tilting modules \cite{So2, Br2}
our second hypothesis above assures that $(U_n(f):K_n(g))\ge 1$.
\end{proof}
\subsection{Tilting modules with short Verma flags}

We recall the following variant of \cite[Corollary 4.27]{Br},
which follows from the fact that $E^{(r)}_a, F^{(r)}_a$ are exact
functors and are both left and right adjoint to each other.

\begin{lemma}\label{sum:tilt}
Let $U$ be a tilting module in $\mathcal O_{\m|n}^+$.  Then
$X_a^{(r)}U$ is a direct sum of tilting modules, where $X=E,F$.
\end{lemma}

\begin{lemma}\label{tiltVerma}
Let $\la \in X_+^{\m|n}$ be atypical. Then the parabolic Verma
module $K_n(\la)$ is not a tilting module.
\end{lemma}

\begin{proof}
For atypical $\la$, $K_n(\la)$ is reducible since the Kac module
(which is the parabolic Verma module with respect to the parabolic
subalgebra whose Levi subalgebra is $\gl(m)\oplus\gl(n)$) as its
quotient is reducible. If $K_n(\la)$ were a tilting module, then by
Coroallry \ref{tiltingdual} we have $K_n(\la)=K_n(\la)^\tau$. But
this is impossible since $K_n(\la)$ is reducible and hence cannot
have isomorphic socle and cococle.
%Such a reducible $K_n(\la)$ (which has $L_n(\la)$ as its cosocle)
%can never be isomorphic to a dual Verma module (which has to have
%$L_n(\la)$ as its socle), and then it cannot be a tilting module by
%Proposition~\ref{flagExt}.
\end{proof}

\begin{proposition}\label{3terms}
Let $h,f\in\Z^{\m|n}_+$ be atypical. Suppose that (i) $U_f=XU_h$
for a product $X$ of $E_a^{(r)}$ and $F^{(r)}_a$ with varying $a$
and $r$, (ii) $i[U_n(h)]=U_h(1)$, and (iii) $U_f$ has at most
three monomial terms. Then $X U_n(h) =U_n(f)$, and
$i[U_n(f)]=U_f(1)$.
\end{proposition}

\begin{proof}
By Theorem~\ref{superChe=tran} and the assumptions (i-ii), we have
that $i[XU_n(h)] =Xi[U_n(h)] =XU_h(1) =U_f(1)$. It follows from this
and Lemma~\ref{sum:tilt} that there is a summand of $XU_n(h)$
isomorphic to $U_n(f)$. So the proposition follows by showing that
$XU_n(h)$ is indecomposable. By the assumption (iii) and Theorem
\ref{superChe=tran}, $XU_n(h)$ has a Verma flag of length at most
three. The weight $g$ in any Verma $K_n(g)$ appearing in a Verma
flag of $XU_n(h)$ must be atypical like $f$. Thus $K_n(g)$ is not
tilting by Lemma~\ref{tiltVerma}, and hence $XU_n(h)$ has to be
indecomposable.
\end{proof}

\subsection{The typical case}

The next proposition is a variant of \cite[Lemma 2.25]{Br} and
\cite[Theorem 4.31]{Br}.  It can be proved by modifying the
arguments therein, using now Theorem \ref{red-KL-finite}.

\begin{proposition}\label{typical}
Let $f\in \Z^{\m+n}_+$ be typical and let $f_{\m}$ denote the
restriction of $f$ to the set $I(m|0)$. We have
\begin{itemize}
\item[(i)] $U_f=U_{f_{\m}}\otimes w_{f(1)}\wedge\cdots\wedge
w_{f(n)}$, where $U_{f_\m}$is the corresponding canonical basis
element in $\mathcal E^{\m+0}$.

\item[(ii)] The linear map $i$ sends $[U_n(f)]$ to $U_f(1)$.
\end{itemize}
\end{proposition}

\subsection{The regular case}

We introduce a {\em Regularity} Condition (R) on $f \in
\Z_+^{m_1,m_2|n}$:
\begin{itemize}
\item[(R)] If $f(i)=f(j)=a$ for some $a\in\Z$ and $-m_1-m_2\le
i<0<j$, then there exists no $k\in I(m_1+m_2|n)\backslash\{i,j\}$
with $f(k)=a-1$ or $f(k)=a$.
\end{itemize}

\begin{theorem}\label{formula:canonical}
Suppose that $f\in\Z^{m_1,m_2| n}_+$ satisfies Condition (R). Then
we have
\begin{itemize}
\item[(i)] $U_f=\sum_{\theta\in\{0,1\}^{\# f}}
\sum_{\Sigma\subseteq\Sigma^+(f)} q^{|\theta| +|\Sigma|}
K_{f^{\texttt {L}_\theta}_\Sigma}$ in $\widehat{\mathcal
E}^{m_1,m_2|n}$,

\item[(ii)] $[U_n(f)]=\sum_{\theta\in\{0,1\}^{\#
f}}\sum_{\Sigma\subseteq\Sigma^+(f)} [K_n(f^{\texttt
{L}_\theta}_\Sigma)]$ in $G(\mathcal O^+_{m_1,m_2|n})$,

\item[(iii)] the tilting module $U_n(f)$ is $\tau$-self-dual and
it has a simple cosocle $L_n(\tilde{f})$, where $\tilde{f}
=f^{\texttt {L}_{(1,\ldots, 1)}}_{\Sigma^+(f)}$ is the minimal
weight in a Verma flag of $U_n(f)$.
\end{itemize}
In particular, a Verma flag of $U_n(f)$ is multiplicity-free and
has length $2^{|\Sigma^+(f)|+\#f}$.
\end{theorem}

\begin{proof}
Let $\#f=k$ and $\{(i_1|j_1),\cdots,(i_k|j_k)\}$ be the set of all
pairs of $f$ with $f(i_t)=f(j_t)$ for $1\le t\le k$, where
$0<j_1<\cdots<j_k$. Since $f$ satisfies (R), we have
$\Sigma^+(f^{\texttt{L}_\theta}) =\Sigma^+(f)$ for every
$\theta\in\{0,1\}^{\#f}$, and moreover,
$\texttt{L}_{i_s,j_s}\circ\texttt{L}_{i_t,j_t}=\texttt{L}_{i_t,j_t}\circ\texttt{L}_{i_s,j_s}$.

Take $({i}_k|{j}_k)$ with $f({i}_k)=f({j}_k)={a}_k$. Assume
without loss of generality that $f$ is of the form
$$(\cdots {a}_k\cdots|\cdots|\cdots {a}_k\cdots).$$
(We omit the parallel proof when $f$ is of the form
$(\cdots|\cdots {a}_k\cdots|\cdots {a}_k \cdots)$.)

We prove (i) by induction on the atypicality number $\#f$. By
Proposition~\ref{typical}, the case $\#f=0$ boils down to
\cite[Theorem~4.25]{CWZ}.

Let $h$ be defined by $h({j}_k)={a}_k-1$ and $h(s)=f(s)$, for all
$s\not={j}_k$. Note that $\#h=\#f-1$ and
$|\Sigma^+(h)|=|\Sigma^+(f)|$, and the induction assumption gives
an explicit formula for $U_h$ in $2^{|\Sigma^+(f)|+\#f-1}$
monomial terms. Set $X=E_{{a}_k-1}$. Then $X U_h$ is clearly
bar-invariant and by a direct calculation is equal to the
right-hand side in (i), hence it has to coincide with $U_f$ by
definition of canonical basis. This proves (i).

We prove (ii) and (iii) together in two inductive steps: (1)
induction on the atypicality number $\#f$ to reduce to the case
when $\#f=0$; (2) in the case when $\#f=0$, induction on the
cardinality $|\Sigma^+(f)|$. In the initial case when
$\#f=|\Sigma^+(f)|=0$, $f$ is minimal in super Bruhat ordering and
$K_n(f)$ is irreducible, and hence $U_n(f) =K_n(f)$ has a simple
cosocle. The arguments (which are based on Method One of the proof
of \cite[Theorem 4.37]{Br}) for these two steps are completely
analogous, and we will only present the inductive step (1) on
$\#f$ in detail below.

By (R), for each $g\preceq f$ we clearly have $F_{a_k-1}^2K_g=0$,
hence $F^{2}_{a_k-1}K_n(g)=0$ and then $F^{2}_{a_k-1}L_n(g)=0$. It
follows by Theorem~\ref{cr:simplicity} that $F_{a_k-1}L_n(g)$ is
irreducible or zero depending on whether or not $F_{a_k-1}K_g$ is
zero (or equivalently, depending on whether $(a_k-1)$-string of the
underlying crystal graph has length $1$ or $0$). Suppose that
\begin{eqnarray} \label{adjoint}
{\rm Hom}_{{{\bf m}|n}}(X U_n(h),L_n(g))\cong {\rm Hom}_{{{\bf
m}|n}}(U_n(h),F_{a_k-1}L_n(g))
\end{eqnarray}
is nonzero for some $g\preceq f$. By the inductive assumption, the
tilting module $U_n(h)$ has a simple cosocle $L_n(\tilde{h})$.
Thus, $F_{a_k-1}L_n(g)=L_n(\tilde{h})$ by
Theorem~\ref{cr:simplicity}. Hence $\tilde{F}_{a_k-1} g
=\tilde{h}, $ and thus $g= \tilde{E}_{a_k-1} \tilde{h}$ since the
$(a_k-1)$-string of the underlying crystal graph is of length $1$,
where $\tilde{E}_{a_k-1}, \tilde{F}_{a_k-1}$ denote the Kashiwara
(crystal) operators corresponding to $E_{a_k-1}, F_{a_k-1}$.  One
checks that $\tilde{E}_{a_k-1} \tilde{h} =\tilde{f}$. Hence $X
U_n(h)$ has a simple cosocle $L_n(\tilde{f})$ and in particular is
indecomposable. This proves (iii).

Now by the induction assumption and Theorem~\ref{superChe=tran},
we have
$$i[X U_n(h)] =X i[U_n(h)] =X U_h(1) =U_f(1).
$$
It follows by Lemma~\ref{sum:tilt} and the indecomposability of $X
U_n(h)$ that $U_n(f) =X U_n(h)$. Together with (i), this proves
(ii).
\end{proof}

\begin{remark}
Setting $n=0$, the proof of Theorem~\ref{formula:canonical} gives
a purely {\em algebraic} proof of the Kazhdan-Lusztig conjecture
for the parabolic category $\mathcal O_+^{m_1,m_2}$ of
$\gl(m_1+m_2)$-modules (compare with Theorem~\ref{th:multiCan}).
\end{remark}

\begin{remark}
%Let $\Z^{m_1+m_2|n_1+n_2}_+$ denote the set of $\Z$-valued
%functions $f$ with $f(-m_1-m_2)>\cdots>f(-m_2-1)$,
%$f(-m_2)>\cdots>f(-1)$, $f(n_1)<\cdots<f(n_1)$ and
%$f(n_1+1)<\cdots<f(n_2)$.
Recall from Remark~\ref{rem:general} that $\Z^{m_1,m_2|n_1,n_2}_+$
parameterizes the bases for the space $\mathcal
E^{m_1,m_2|n_1,n_2}$. Suppose that $f$ satisfies the following
condition:
\begin{itemize}
\item[(RR)] If $f(i)=f(j)=a$ for some $a\in\Z$ with $i<0<j$, then
there exists no $k\in I(m_1+m_2| n_1+n_2)\backslash\{i,j\}$ with
$f(k)=a-1$ or $f(k)=a$.
\end{itemize}
Denote by $\Sigma^+(-f_{34})$ the set of positive pairs  of
$-f_{34}$. Since $f$ satisfies Condition (RR), we have $\Sigma^+
(f_{12}) =\Sigma^+ (f^{\texttt{L}_\theta}_{12})$ and
$\Sigma^+(-f_{34})=\Sigma^+ (-f^{\texttt{L}_\theta}_{34}))$, for
any $\theta\in\{0,1\}^{\#f}$. The argument for Theorem
\ref{formula:canonical} can be modified easily to establish the
following formula for the canonical basis:
\begin{equation*}
U_f=\sum_{\theta\in\{0,1\}^{\# f}}
\sum_{\Sigma\subseteq\Sigma^+(f^{\texttt L_\theta}_{12})} \;\;
\sum_{\Gamma\subseteq\Sigma^+ (-f^{\texttt{L}_\theta}_{34})}
q^{|\theta|+|\Sigma|+|\Gamma|} K_{f^{\texttt
L_\theta}_{\Sigma,\Gamma}}.
\end{equation*}
Here ${f^{\texttt L_\theta}_{\Sigma,\Gamma}}$ denotes the function
obtained from $f^{\texttt L_\theta}$ by first interchanging the
values of $f$ at each positive pair in $\Sigma$ and $\Gamma$, and
then taking the unique conjugate under $S_{m_1}\times
S_{m_2}\times S_{n_1}\times S_{n_2}$ in $\Z^{m_1,m_2|n_1,n_2}_+$.

The corresponding multiplicity-free formula holds for the tilting
module in the category $\mathcal O_+^{m_1,m_2|n_1,n_2}$ (see
Remark~\ref{rem:genO}).
%\end{theorem}
\end{remark}

\begin{remark}
For $f$ satisfying the condition (R) or (RR), the formulae for
$U_f$ and $U_n(f)$ above support Conjecture~\ref{conj:BrKL}.
\end{remark}

\section{The category $\mathcal O_{1,1|n}^+$ of
$\mathfrak{gl}(2|n)$-modules} \label{canonical:11n}

In this section, we analyze completely the case for $\m =(1,1)$.
We find explicit formulas for canonical basis in $\mathcal
E^{1,1|n}$, and establish the parabolic Brundan
Conjecture~\ref{conj:BrKL} for the category $\mathcal
O^+_{1,1|n}$.

\subsection{A procedure for canonical basis}

For $f\in\Z^{1,1|n}_+$, we denote $\Sigma_{f_{13}}=(-2|j)$ if
there exists $j>0$ with $f(-2)=f(j)$, and otherwise set
$\Sigma_{f_{13}}=\emptyset$. Similarly, denote
$\Sigma_{f_{23}}=(-1|j)$ if there exists $j>0$ with $f(-1)=f(j)$,
and otherwise set $\Sigma_{f_{23}}=\emptyset$. If
$\Sigma_{f_{13}}\cup\Sigma_{f_{23}}=\emptyset$, then $\#f=0$.
Below we give a procedure to reduce any $f$ such that
$\Sigma_{f_{13}}\cup\Sigma_{f_{23}} \not =\emptyset$ to $g$ such
that $\Sigma_{g_{13}}\cup\Sigma_{g_{23}} =\emptyset$.

\begin{procedure}\label{procedure:11n}
Let $f\in\Z_+^{1+1|n}$ be such that
$\Sigma_{f_{13}}\cup\Sigma_{f_{23}}\not=\emptyset$.
\begin{itemize}
\item[Step 1] If $\Sigma_{f_{13}}=\emptyset$ go to Step 4.
Otherwise go to Step 2.

\item[Step 2] If $f(-2)\not=f(-1)$, go to Step 3. Otherwise let
$h$ be the function obtained from $f$ by setting
$h(-1)=h(-2)=f(-2)-1$ and $h(i)=f(i)$, for $i>0$. Let
$X=F^{(2)}_{f(-2)-1}$. Stop.

\item[Step 3] Let $h$ be the function obtained from $f$ by setting
$h(-2)=f(-2)-1$ and $h(i)=f(i)$, for $i\not=-2$. Let
$X=F_{f(-2)-1}$. Stop.

\item[Step 4] If $f(-2)=f(-1)-1$ go to Step 5.  Otherwise we let
$h$ be the function defined by $h(-1)=f(-1)-1$ and $h(s)=f(s)$,
for $s\not=-1$. Let $X=F_{f(-1)-1}$. Stop.

\item[Step 5] If there exists $i>0$ with $f(i)=f(-2)-1$, go to
Step 6.  Otherwise go to Step 3.

\item[Step 6] Let $j>0$ with $f(-1)=f(j)$. Let $k>1$ be the
smallest integer such that $f(j)-k\not=f(j-k+1)$.  Let $h$ be
defined by $h(j-k)=f(j-k)-1$ and $h(s)=f(s)$, for $s\not=j-k$. Let
$X=F_{f(j-k)-1}$. Stop.
%We let $h$ be obtained from $f$ by $h(j-i)=f(j-i)-1$, for all
%$i=1,\cdots,k-2$, $h(-2)=f(-2)-1$, $h(-1)=f(-1)-1$, and $h(s)=f(s)$,
%otherwise. $X=E_{a-k}E_{a-k+1}\cdots E_{a-4}E_{a-3}F_{a-2}F_{a-1}$.
\end{itemize}
\end{procedure}
As can be seen case by case below, repeated application of the
above procedure will produce in finite steps an element $g$ such
that $\#g=0$.

\begin{theorem}\label{canonical11n}
Let $f$ be such that $\Sigma_{f_{13}}\cup\Sigma_{f_{23}} \not
=\emptyset$. Let $X$ and $h$ be as defined in
Procedure~\ref{procedure:11n}. Then we have
\begin{enumerate}
 \item[(i)]
$U_f=XU_h$ in $\mathcal E^{1,1|n}$,
 \item[(ii)]
 $XU_n(h)=U_n(f)$ in $\mathcal O^+_{1,1|n}$,
\item[(iii)] $i[U_n(f)]=U_f(1)$, \item[(iv)] the tilting module
$U_n(f)$ is $\tau$-self-dual and it has a simple cosocle.
\end{enumerate}
\end{theorem}

The proof of  Theorem \ref{canonical11n} will be postponed to the
following subsections. We note the following immediate
consequence.
\begin{theorem}
The Conjecture~\ref{conj:BrKL} for the category $\mathcal
O^+_{1,1|n}$ holds.
\end{theorem}

\begin{proof}
The case of typical weights is taken care by
Proposition~\ref{typical}. The case of atypical weights follows
from Theorem \ref{canonical11n}.
\end{proof}

Below as usual we will denote by $\cdots$ an expression with no
$a$ or $a-1$. The proof for Theorem~\ref{canonical11n} is done
case by case, and the main argument in most cases is the same as
the one for Theorem~\ref{formula:canonical}. In particular, a main
point of the argument is to check if the assumption in
Theorem~\ref{cr:simplicity} is also satisfied.

\subsection{Proof of Theorem~\ref{canonical11n}, I}

In the subsection, we consider the case when
$|\Sigma_{f_{13}}|=1$. Here we have the following possibilities:

\begin{itemize}
\item[(i)] $f=(a|a|\cdots a\cdots)$ \item[(ii)] $f=(a|a|\cdots
a-1,a\cdots)$ \item[(iii)] $f=(a|a-1|\cdots a-1,a\cdots)$
\item[(iv)] $f=(a|\cdots|\cdots a\cdots)$

\item[(v)] $f=(a|a-1|\cdots a\cdots)$

\item[(vi)] $f=(a|\cdots|\cdots a-1,a\cdots)$
\end{itemize}

In (i) we set $h=(a-1|a-1|\cdots a\cdots)$ and $X=F^{(2)}_{a-1}$.
We note that $h$ is a typical weight and hence we have
$U_{(a-1|a-1|\cdots a\cdots)}=K_{(a-1|a-1|\cdots a\cdots)}$. Thus
\begin{equation*}
U_{(a|a|\cdots a\cdots)}=XU_{(a-1|a-1|\cdots
a\cdots)}=K_{(a|a|\cdots a\cdots)}+qK_{(a|a-1|\cdots
a-1\cdots)}+q^2K_{(a|a|\cdots a-1\cdots)}.
\end{equation*}
It follows now from Proposition~\ref{3terms} that
$U_n(f)=XU_n(h)$. Alternatively we can show this using the same
type of argument as in the proof of
Theorem~\ref{formula:canonical} as follows. Suppose that $g\preceq
f$ in the Bruhat ordering. Then $g$ must be of the form
$(a|a|\cdots a\cdots)$, $(a-1|a|\cdots a-1\cdots)$, $(a|a-1|\cdots
a-1\cdots)$, $(\cdots|a|\cdots )$ or $(a|\cdots|\cdots )$. It is
easy to see that $Y^{3}L_n(g)=0$, where $Y=E_{a-1}$. Thus in this
case the assumptions of Theorem~\ref{cr:simplicity} is satisfied
and hence $Y^{(2)}L_n(g)$ is irreducible. Therefore the same
argument for Theorem~\ref{formula:canonical} can be applied to
show that $XU_n((a-1|a-1|\cdots a\cdots))$ has a simple cosocle
and is isomorphic to $U_n((a|a|\cdots a\cdots))$.

In (ii) we set $h=(a-1|a-1|\cdots a-1,a\cdots)$ and
$X=F^{(2)}_{a-1}$. Now any weight less than $h$ is of the form
$(\cdots|a-1|\cdots a\cdots)$ or $(\cdots|a-1|\cdots a\cdots)$.
Thus upon application of $X$ the $q$-power is preserved.  It
follows therefore that $XU_{(a-1|a-1|\cdots
a-1,a\cdots)}=U_{(a|a|\cdots a-1,a\cdots)}$. Furthermore it is
easy to check that if $g\preceq f$, then $Y^{3}g=0$, and hence the
assumption in Theorem~\ref{cr:simplicity} is also satisfied. Thus
$XU_n((a-1|a-1|\cdots a-1,a\cdots))=U_n((a|a|\cdots
a-1,a\cdots))$.

In (iii) set $h=(a-1|a-1|\cdots a-1,a\cdots)$ and $X=F_{a-1}$.  If
$g\prec h$, then $g$ is of the form $(\cdots|a-1|\cdots a\cdots)$
or $(a-1|\cdots|\cdots a\cdots)$ and so we see that $U_{f}=X U_h$.
If $g\preceq f$, then $g$ is of the form $(a|a-1|\cdots
a-1,a\cdots)$, $(a-1|a|\cdots a-1,a\cdots)$, $(\cdots|a-1|\cdots
a-1\cdots)$, $(a-1|\cdots|\cdots a-1\cdots)$, $(\cdots|a|\cdots
a\cdots)$ or $(a|\cdots|\cdots a\cdots)$. Let $Y=E_{a-1}$ and we
see that $Y^{2}g=0$ satisfying the assumption of
Theorem~\ref{cr:simplicity}.

In (iv) set $h=(a-1|\cdots|\cdots a\cdots)$ and $X=F_{a-1}$. If
$g\prec h$, then $g$ is of the form $(a-1|\cdots|\cdots a\cdots)$,
$(\cdots|a-1|\cdots a\cdots)$, or $(a-1|a-1|\cdots a-1,a\cdots)$.
>From this we see that $XU_h=U_f$. If $g\preceq f$, then $g$ is of
the form $(a|\cdots|\cdots a\cdots)$, $(a-1|\cdots|\cdots
a-1\cdots)$, $(a|a-1|\cdots a-1,a\cdots)$, $(a-1|a|\cdots
a-1,a\cdots)$ or $(\cdots|a|\cdots a\cdots)$.  So we have
$Y^{2}g=0$, for $Y=E_{a-1}$.

In (v) set $h=(a-1|a-1|\cdots a\cdots)$, while in (vi) set
$h=(a-1|\cdots|\cdots a-1, a\cdots)$.  Here $X=F_{a-1}$. In either
case we have $XU_h=U_f$ and if $g\preceq f$, then $Y^{2}g=0$, for
$Y=E_{a-1}$.

\subsection{Proof of Theorem \ref{canonical11n}, II}
In this subsection, we consider the case when
$|\Sigma_{f_{13}}|=0$ and $|\Sigma_{f_{23}}|=1$. Here we have the
following possibilities.

\begin{itemize}
\item[(i)] $f=(\cdots|a|\cdots a\cdots)$

\item[(ii)] $f=(\cdots|a|\cdots a-1, a\cdots)$

\item[(iii)] $f=(a-1|a|\cdots a\cdots)$.
\end{itemize}

In (i) we set $h=(\cdots|a-1|\cdots a\cdots)$, while in (ii) we
set $h=(\cdots|a-1|\cdots a-1, a\cdots)$.  In both cases
$X=F_{a-1}$ and it is easy to see that in either case we have
$XU_h=U_f$. In (i) if $g\preceq f$, then $g$ is of the form
$(\cdots|a|\cdots a\cdots)$, $(a|\cdots|\cdots a\cdots)$,
$(\cdots|a-1|\cdots a-1\cdots)$, $(a-1|\cdots|\cdots a-1\cdots)$,
$(a-1|a|\cdots a-1,a\cdots)$ or $(a|a-1|\cdots a-1,a\cdots)$.
Clearly $Y^{2}g=0$, for $Y=E_{a-1}$.

In (ii) if $g\preceq f$, then $g$ is of the form $(\cdots|a|\cdots
a-1,a\cdots)$ or $(a|\cdots|\cdots a-1,a\cdots)$.  Also we have
$Y^{2}g=0$.  So in both cases the hypothesis of
Theorem~\ref{cr:simplicity} is satisfied, and thus
$U_n(f)=XU_n(h)$.

Finally for (iii) we consider first the case $f=(a-1|a|\cdots
a\cdots)$, where $a-2$ is not contained in $\cdots$.  We set
$X=F_{a-2}$ and $h=(a-2|a|\cdots a\cdots)$. It is easy to check
that $U_f=XU_h$. Next let $g=(a-2|a-1|\cdots a\cdots)$ and
$X'=F_{a-1}$.  Again it is easy to see that $U_h=X'U_g$, so that
we have $U_f=XX'U_g$.  Now $g$ is typical and hence $U_g=K_g$.
Thus we obtain
\begin{equation*}
U_f=XX'K_g=K_{(a-1|a|\cdots a\cdots)}+qK_{(a-1|a-1|\cdots
a-1\cdots)}+qK_{(a-2|a-1|\cdots a-2\cdots)}
\end{equation*}
By Proposition~\ref{3terms} $XX'U_n(g)$ is isomorphic to $U_n(f)$.
Now $X'U_n(g)$ has a parabolic Verma flag of length two, and hence
by Proposition~\ref{3terms} again, we see that $X'U_n(g)=U_n(h)$.
Thus we conclude that $XU_n(h)=U_n(f)$.

We will use ${x}\sim y$ to denote the sequence of integers from
$x$ to $y$. Suppose that $f=(a-1|a|\cdots, (a-k+1)\sim
(a-2),a\cdots)$ and $a-k$ is not in $\cdots$, where $k \ge 3$. We
consider the following sequence
\begin{align*}
&(a-1|a|\cdots, (a-k+1)\sim
(a-2),a\cdots)\stackrel{E_{a-k}}{\leftarrow}\\
&(a-1|a|\cdots, a-k, (a-k+2)\sim (a-2),a\cdots)
\stackrel{E_{a-k+1}}{\leftarrow} \\
&(a-1|a|\cdots
a-k,a-k+1,(a-k+3)\sim (a-2),a\cdots)\stackrel{E_{a-k+2}}{\leftarrow}\cdots\\
&\cdots \stackrel{E_{a-3}}{\leftarrow} (a-1|a|\cdots,
(a-k)\sim(a-3),a\cdots)=g
\end{align*}

\begin{lemma} Let $x$, $y$ and $a$ be distinct and $x,y>a$.
Let $f=(x|y|\cdots a\cdots y\cdots)$, where $\cdots$ denotes an
expression with no $x$, $a$ and $a-1$.  Let $h=(x|y|\cdots
a-1\cdots y\cdots)$ and $X=E_{a-1}$.  Then $XU_h=U_f$.
\end{lemma}

\begin{proof}
Any $g\preceq h$ is of the form $(x|y|\cdots a-1\cdots y\cdots)$,
$(y|x|\cdots a-1\cdots y\cdots)$, $(x|a|\cdots a-1\cdots a\cdots)$
or $(a|x|\cdots a-1\cdots a\cdots)$.
\end{proof}

Thus we have
\begin{equation*}
E_{a-k}E_{a-k+1}\cdots E_{a-4}E_{a-3}U_g=U_f.
\end{equation*}
Now $U_g=K_{(a-1|a|\cdots a\cdots)}+qK_{(a-1|a-1|\cdots
a-1\cdots)}+qK_{(a-2|a-1|\cdots a-2\cdots)}$.  A simple
calculation shows that
\begin{equation*}
U_f=K_{(a-1|a|\cdots a\cdots)}+qK_{(a-1|a-1|\cdots
a-1\cdots)}+qK_{(a-k|a-1|\cdots a-k\cdots)}.
\end{equation*}
Now Proposition~\ref{3terms} shows that $U_n(f)$ has a Verma flag
consisting of parabolic Verma modules of these three highest
weights.  Now every $E_{a-k+i}\cdots E_{a-4}E_{a-3}U_g$, for every
$i>0$, contains three monomials, and thus $E_{a-k+i}\cdots
E_{a-4}E_{a-3}U_n(g)$ is a tilting module by
Proposition~\ref{3terms}. In particular $E_{a-k+1}\cdots
E_{a-4}E_{a-3}U_n(g)=U_n(h)$ and hence $E_{a-k}U_n(h)=U_n(f)$.

\subsection{Formulas for canonical basis elements}
\label{subsec:formula11n}

In this subsection, we provide a complete list of formulas for the
canonical basis elements in $\mathcal E^{1,1|n}$ (except the
trivial case when $f$ is typical). They are computed using
Procedure~\ref{procedure:11n}, and thus by Theorem
\ref{canonical11n} we find explicit Verma flag weights of the
tilting modules in the category $\mathcal O^+_{1,1|n}$ as well.

Recall that we use ${x}\sim a$ to denote the sequence of integers
from $x$ to $a$, and we shall use $\widehat{x}\sim a$ the sequence
of integers from $x+1$ to $a$. We assume $c>a>b$.
\begin{align*}
&  {\text{Atypicality } 2:} \\
(A1)
 &\;\; U_{(a|b|\cdots\widehat{x}\sim b\sim
a\cdots)}=K_{(a|b|\cdots\widehat{x}\sim b\sim
a\cdots)}+qK_{(b|a|\cdots\widehat{x}\sim b\sim
a\cdots)}+qK_{(a|x|\cdots{x}\sim \widehat{b}\sim a\cdots)}\\
&\qquad\qquad\qquad  +q^2K_{(b|x|\cdots{x}\sim b\sim
\widehat{a}\cdots)} +q^2K_{(x|a|\cdots{x}\sim\widehat{b}\sim
a\cdots)} +q^3K_{(x|b|\cdots{x}\sim b\sim
\widehat{a}\cdots)}.\\
(A2)
 &\;\; U_{(a|b|\cdots \widehat{y}\sim b\cdots\widehat{x}\sim
a\cdots)}=K_{(a|b|\cdots \widehat{y}\sim b\cdots\widehat{x}\sim
a\cdots)}+qK_{(x|b|\cdots \widehat{y}\sim b\cdots{x}\sim
\widehat{a}\cdots)}+qK_{(a|y|\cdots {y}\sim
\widehat{b}\cdots\widehat{x}\sim a\cdots)}\\
&\qquad\qquad\qquad +qK_{(b|a|\cdots \widehat{y}\sim
b\cdots\widehat{x}\sim a\cdots)}+q^2K_{(b|x|\cdots \widehat{y}\sim
b\cdots{x}\sim \widehat{a}\cdots)}+q^2K_{(y|a|\cdots {y}\sim
\widehat{b}\cdots\widehat{x}\sim a\cdots)}\\
&\qquad\qquad\qquad +q^2K_{(x|y|\cdots {y}\sim
\widehat{b}\cdots{x}\sim \widehat{a}\cdots)}+q^3K_{(y|x|\cdots
{y}\sim
\widehat{b}\cdots{x}\sim \widehat{a}\cdots)}.\\
(A3)
 &\;\; U_{(b|a|\cdots\widehat{y}\sim\widehat{x}\sim b\sim
a\cdots)}=K_{(b|a|\cdots\widehat{y}\sim\widehat{x}\sim b\sim
a\cdots)}+qK_{(x|a|\cdots\widehat{y}\sim{x}\sim \widehat{b}\sim
a\cdots)}+qK_{(b|x|\cdots\widehat{y}\sim{x}\sim b\sim
\widehat{a}\cdots)}\\
&\qquad\qquad\qquad +qK_{(y|b|\cdots{y}\sim\widehat{x}\sim b\sim
\widehat{a}\cdots)}+q^2K_{(x|b|\cdots\widehat{y}\sim{x}\sim b\sim
\widehat{a}\cdots)}+q^2K_{(y|x|\cdots{y}\sim{x}\sim
\widehat{b}\sim \widehat{a}\cdots)}.\\
(A4) &\;\; U_{(b|a|\cdots\widehat{y}\sim b\cdots\widehat{x}\sim
a\cdots)}=K_{(b|a|\cdots\widehat{y}\sim b\cdots\widehat{x}\sim
a\cdots)}+qK_{(y|a|\cdots{y}\sim\widehat{b}\cdots\widehat{x}\sim
a\cdots)}\\
&\qquad\qquad\qquad\qquad  +qK_{(b|x|\cdots\widehat{y}\sim
b\cdots{x}\sim \widehat{a}\cdots)}+q^2K_{(y|x|\cdots{y}\sim
\widehat{b}\cdots{x}\sim \widehat{a}\cdots)}.
\end{align*}

\begin{align*}
&  {\text{Atypicality } 1} \quad (b<x<a \text{ is assumed below}): \\
(B1)
 &\; U_{(a|c|\cdots\widehat{x}\sim
a\cdots)}=K_{(a|c|\cdots\widehat{x}\sim
a\cdots)}+qK_{(x|c|\cdots{x}\sim
\widehat{a}\cdots)}.\\
(B2)
 &\; U_{(a|b|\cdots\widehat{x}\sim
a\cdots)}=K_{(a|b|\cdots\widehat{x}\sim
a\cdots)}+qK_{(b|a|\cdots\widehat{x}\sim
a\cdots)}+qK_{(x|b|\cdots{x}\sim \widehat{a}\cdots)}
+q^2K_{(b|x|\cdots{x}\sim
\widehat{a}\cdots)}.\\
(B3)
 &\; U_{(a|x|\cdots\widehat{x}\sim
a\cdots)}=K_{(a|x|\cdots\widehat{x}\sim
a\cdots)}+qK_{(x|a|\cdots\widehat{x}\sim
a\cdots)}+q^2K_{(x|x|\cdots{x}\sim
\widehat{a}\cdots)}.\\
(B4)
 &\; U_{(c|a|\cdots\widehat{x}\sim
a\cdots)}=K_{(c|a|\cdots\widehat{x}\sim
a\cdots)}+qK_{(a|c|\cdots\widehat{x}\sim
a\cdots)}+qK_{(c|x|\cdots{x}\sim
\widehat{a}\cdots)}+q^2K_{(x|c|\cdots{x}\sim
\widehat{a}\cdots)}.\\
(B5)
 &\; U_{(b|a|\cdots\widehat{x}\sim
a\cdots)}=K_{(b|a|\cdots\widehat{x}\sim
a\cdots)}+qK_{(b|x|\cdots{x}\sim
\widehat{a}\cdots)}.\\
(B6)
 &\; U_{(x|a|\cdots\widehat{y}\sim\widehat{x}\sim
a\cdots)}=K_{(x|a|\cdots\widehat{y}\sim\widehat{x}\sim
a\cdots)}+qK_{(x|x|\cdots\widehat{y}\sim{x}\sim
\widehat{a}\cdots)}+qK_{(y|x|\cdots{y}\sim\widehat{x}\sim
\widehat{a}\cdots)}.\\
(S)
 &\; U_{(a|a|\cdots \widehat{x}\sim a\cdots)}=K_{(a|a|\cdots
\widehat{x}\sim a\cdots)}+qK_{(a|x|\cdots
x\sim\widehat{a}\cdots)}+q^2K_{(x|a|\cdots
{x}\sim\widehat{a}\cdots)}.\\
& \text{\quad (This last weight is special in the sense that it
has three identical values.)}
\end{align*}

\subsection{Super duality: a weak version}

The following weak version of Conjecture \ref{conj:duality} holds
in the case ${\bf m}=(1,1)$.

\begin{theorem}
The categories $\mathcal O^+_{1,1|\infty}$ and $\mathcal
O^+_{(1,1)+\infty}$ admit isomorphic Kazhdan-Lusztig theories. In
particular, for $f,g\in \Z^{(1,1)+\infty}_+$ we have
\begin{eqnarray*}
(\mathcal U(f):\mathcal K(g)) &=&(U(f^\natural):K(g^\natural)),  \\
{[}\mathcal K(f):\mathcal L(g)] &=& [K(f^\natural):L(g^\natural)].
\end{eqnarray*}
\end{theorem}

\begin{proof}
In light of Theorem \ref{canonical11n}~(iii), Theorem
\ref{tiltInf} and Corollary \ref{aux43} we see that
\begin{equation*}
u_{gf}(1)=u_{g^{(n)}f^{(n)}}(1)=(U_n(f^{(n)}):K_n(g^{(n)}))=(U(f):K(g)),\quad
n\gg 0.
\end{equation*}
By Theorem \ref{tiltInf:red}, Theorem \ref{th:multiCan} and
Proposition \ref{commutativity-reduct} we have for
$f',g'\in\Z^{(1,1)+\infty}_+$
\begin{equation*}
\mathfrak u_{g',f'}(1)=\mathfrak
u_{g'^{(n)},f'^{(n)}}(1)=(\mathcal U_n(f'^{(n)}):\mathcal
K_n(g'^{(n)}))=(\mathcal U(f'):\mathcal K(g')),\quad n\gg 0.
\end{equation*}
Now the first identity in the theorem follows by
Theorem~\ref{correspondence}.

The second identity in the theorem follows by Remarks
\ref{dualUK:super} and \ref{dualUK}, Theorem \ref{red-KL-finite}
and (\ref{eq:tiltduality}), together with the corresponding
compatibility of truncation functors on irreducible
representations and truncation maps on dual canonical basis
elements.
\end{proof}
\section{The category $\mathcal O_{m,1|1}^+$ of
$\mathfrak{gl}(m+1|1)$-modules} \label{sec:categorym11}

In this section, we analyze completely the case for $\m =(m,1)$
and $n=1$. We find explicit formulas for canonical basis in
$\mathcal E^{m,1|1}$, and establish the parabolic BKL
Conjecture~\ref{conj:BrKL} for the category $\mathcal
O^+_{m,1|1}$.

\subsection{A procedure}
%Let $I(m+1|1)=\{-m-1,\cdots,-1|1\}$.
Denote $\Sigma_{f_{13}}=(i|1)$ if there exists $i<-1$ with
$f(i)=f(1)$, and otherwise set $\Sigma_{f_{13}}=\emptyset$. Also
denote $\Sigma_{f_{23}}=(-1|1)$ if $f(-1)=f(1)$, and otherwise set
$\Sigma_{f_{23}}=\emptyset$. If
$\Sigma_{f_{13}}\cup\Sigma_{f_{23}}=\emptyset$, then $\#f=0$.

\begin{procedure}\label{procedure:m11}
Let $f\in\Z^{m,1|1}_+$ be such that
$\Sigma_{f_{13}}\cup\Sigma_{f_{23}}\not=\emptyset$. \item[Step 1]
If $\Sigma_{f_{13}}=\emptyset$, go to Step 5. Otherwise let
$\Sigma_{f_{13}}=(i|1)$ and go to Step 2.

\item[Step 2] If $i<-1$ and $f(i+1)=f(i)-1$, replace $i$ by $i+1$
and repeat Step 2. Otherwise go to Step 3.

\item[Step 3] If $f(i)=f(-1)$, go to Step 4. Otherwise we set
$h(i)=f(i)-1$ and $h(s)=f(s)$, for $s\not=i$. Let $X=F_{f(i)-1}$.
Stop.

\item[Step 4] Set $h(i)=h(-1)=f(i)-1$, and $h(s)=f(s)$, for
$s\not=i,-1$.  Let $X=F_{f(i)-1}^{(2)}$.  Stop.

\item[Step 5] We have $f(-1)=f(1)$.  If there exists $i<-1$ such
that $f(i)=f(-1)-1$, go to Step 2. Otherwise set $h(-1)=f(-1)-1$,
and $h(s)=f(s)$, for $s\not=-1$.  Let $X=F_{f(-1)-1}$.  Stop.
\end{procedure}

\subsection{Formulas for canonical basis}
\label{subsec:formula:m11}

We will leave the straightforward verification of the following to
the reader.

\begin{proposition}  \label{canonical}
Let $f$ be such that $\Sigma_{f_{13}}\cup\Sigma_{f_{23}} \not
=\emptyset$. Let $X$ and $h$ be as defined in Procedure
\ref{procedure:m11}. Then we have $XU_h=U_f$.
\end{proposition}

Repeated application of Procedure~\ref{procedure:m11} will produce
an element $g$ with $\#g=0$. By Proposition~\ref{typical} we have
$U_g=U_{g_{12}}\otimes w_{g(1)}$. Thus the above procedure
computes all canonical basis elements in $\mathcal E^{m,1|1}$.
Below we present a complete list of formulas for the canonical
basis elements (except the really simple case when $f$ is
typical). We caution that some cases will be missing if $m$ is too
small.

\begin{align*}
& {\bf \text{Atypical cases:}}\\
 (C1)
 &\quad U_{(\cdots a\sim b\sim\widehat{x}\cdots|b|a)}=K_{(\cdots
a\sim b\sim\widehat{x}\cdots|b|a)}+q K_{(\cdots \widehat{a}\sim
b\sim{x}\cdots|b|x)},\quad a>b>x.\\
(C2)
 &\quad U_{(\cdots a \sim\widehat{x}\cdots|x|a)}=K_{(\cdots a
\sim\widehat{x}\cdots|x|a)}+qK_{(\cdots a
\sim\widehat{x+1},x\cdots|x+1|a)} \\
& \qquad\qquad\qquad\qquad  +qK_{(\cdots\widehat{a}
\sim{x}\cdots|x+1|x+1)}+q^2K_{(\cdots \widehat{a}
\sim{x}\cdots|x|x)},\quad a-1>x.\\
(C3)
 &\quad U_{(\cdots a,\widehat{a-1}\cdots|a-1|a)}=K_{(\cdots
a,\widehat{a-1}\cdots|a-1|a)}+q K_{(\cdots
\widehat{a},{a-1}\cdots|a|a)}+
q^2K_{(\cdots \widehat{a},{a-1}\cdots|a-1|a-1)}.\\
(C4)
 &\quad U_{(\cdots
c,\widehat{a}\sim\widehat{x}\cdots|a|a)}=K_{(\cdots
c,\widehat{a}\sim\widehat{x}\cdots|a|a)}+qK_{(\cdots
\widehat{c},{a}\sim\widehat{x}\cdots|c|a)}+qK_{(\cdots
c,\widehat{a}\sim\widehat{x}\cdots|a-1|a-1)}\\
&\qquad\qquad\qquad\qquad +qK_{(\cdots
c,\widehat{a-1}\sim{x}\cdots|a-1|x)}+q^2K_{(\cdots
\widehat{c},{a-1}\sim{x}\cdots|c|x)},\quad a-1>x.\\
(C5)
 &\quad
U_{(\widehat{a}\sim\widehat{x}\cdots|a|a)}=K_{(\widehat{a}\sim\widehat{x}\cdots|a|a)}+qK_{(
\widehat{a}\sim\widehat{x}\cdots|a-1|a-1)}+qK_{(\widehat{a},\widehat{a-1}\sim{x}\cdots|a-1|x)},\;
a-1>x.\\
(C6)
 &\quad
U_{(\widehat{a},\widehat{a-1}\cdots|a|a)}=K_{(\widehat{a},\widehat{a-1}\cdots|a|a)}+qK_{(
\widehat{a},\widehat{a-1}\cdots|a-1|a-1)}.\\
(C7)
 &\quad U_{(\cdots
c,\widehat{a},\widehat{a-1}\cdots|a|a)}=K_{(\cdots
c,\widehat{a},\widehat{a-1}\cdots|a|a)}+qK_{(\cdots
\widehat{c},{a},\widehat{a-1}\cdots|c|a)} \\
& \qquad\qquad\qquad\qquad \;\;  +qK_{(\cdots
c,\widehat{a},\widehat{a-1}\cdots|a-1|a-1)} +q^2K_{(\cdots \widehat{c},{a-1}\cdots|c|a-1)}.\\
(C8)&\quad U_{(\cdots d,\widehat{c}\cdots
a\sim\widehat{x}\cdots|c|a)}=K_{(\cdots d,\widehat{c}\cdots
a\sim\widehat{x}\cdots|c|a)}+q K_{(\cdots d,\widehat{c}\cdots
\widehat{a}\sim{x}\cdots|c|x)}  \\
&\qquad\qquad\qquad\qquad+qK_{(\cdots \widehat{d},{c}\cdots
a\sim\widehat{x}\cdots|d|a)}+ q^2 K_{(\cdots \widehat{d},{c}\cdots
\widehat{a}\sim{x}\cdots|d|x)}
,\quad d>c>a.\\
(C9)&\quad U_{(\cdots a\sim\widehat{x}\cdots
e,\widehat{b}\cdots|b|a)}=K_{(\cdots a\sim\widehat{x}\cdots
e,\widehat{b}\cdots|b|a)} +qK_{(\cdots \widehat{a}\sim{x}\cdots
e,\widehat{b}\cdots|b|x)}\\
&\qquad\qquad\qquad\qquad + qK_{(\cdots a\sim\widehat{x}\cdots
\widehat{e},{b}\cdots|e|a)}+ q^2K_{(\cdots
\widehat{a}\sim{x}\cdots \widehat{e},{b}\cdots|e|x)},\quad x>e>b.
\end{align*}

\begin{align*}
 (T1)
 &\quad U_{(\cdots a\sim\widehat{x}\cdots|a|a)}=K_{(\cdots
a\sim\widehat{x}\cdots|a|a)}+qK_{(\cdots
{a}\sim\widehat{x}\cdots|a-1|a-1)}\\
& \qquad\qquad\qquad\qquad +qK_{(\cdots
a,\widehat{a-1}\sim{x}\cdots|a-1|x)}
+q^2K_{(\cdots \widehat{a}\sim{x}\cdots|a|x)},\quad a-1>x.\\
(T2)
 &\quad U_{(\cdots a,\widehat{a-1}\cdots|a|a)}=K_{(\cdots
a,\widehat{a-1}\cdots|a|a)}+qK_{(\cdots
{a},\widehat{a-1}\cdots|a-1|a-1)}+q^2K_{(\cdots
\widehat{a},a-1\cdots|a|a-1)}.\\
& \text{\quad (In the cases (T1, T2) the weights have three
identical values.)}
\end{align*}

The case of (C8) (respectively (C9)), when no such $d$
(respectively $e$) exists, is obtained by dropping the last two
terms.
\subsection{Structure of tilting modules in $\mathcal O^+_{m,1|1}$}

We shall denote the tilting modules in $\mathcal O^+_{m,1|1}$ by
$U(f)$ et cetera.
%(instead of $U_1(f)$ which are consistent with earlier notation).

\begin{theorem}\label{main:m11}
For any $f\in\Z^{m,1|1}_+$ we have $i([U(f)])=U_f(1)$.
\end{theorem}

\begin{proof}
For typical $f$, this follows from Proposition~\ref{typical}. So
let us now assume that $f$ is atypical.

Each canonical basis element $U_f$ in
Subsection~\ref{subsec:formula:m11} is obtained by applying a
sequence of Chevalley generators dictated by
Procedure~\ref{procedure:m11} to a canonical basis element of
typical weight. Applying the same sequence of translation functors
gives us a sum of tilting modules, denoted by $M(f)$, whose Verma
flag weights are identical to those for the monomials in $U_f$, by
Lemma~\ref{sum:tilt} and Theorem~\ref{superChe=tran}. It follows
by Proposition~\ref{canonical} that $i([M(f)])=U_f(1)$. So it
remains to show that $M(f)=U(f)$. Noting that $U(f)$ is a summand
of $M(f)$, it suffices to prove that $M(f)$ is indecomposable. We
argue case by case using Proposition~\ref{condition:can} and
Proposition~\ref{3terms} as follows.

The indecomposability of $M(f)$ follows from
Proposition~\ref{3terms} if the number of monomials is at most
three. So it remains to check the cases of (T1), (C2), (C4), (C7)
and (C8) and (C9) (in the last cases we only need to consider them
when they have four terms).

For $f$ of the form in (T1), the Verma modules with the first two
weights among four weights in (T1) must lie in the same tilting
module by Proposition~\ref{condition:can}~(ii). Now $M(f)$ is a
direct sum of at most two tilting modules, by Lemma~\ref{tiltVerma}.
If $M(f)$ were a direct sum of two tilting modules, it has to be
$U(f)\oplus U(f^3)$ where $f^3=(\cdots
a,\widehat{a-1}\sim{x}\cdots|a-1|x)$ and $f^4 =(\cdots
\widehat{a}\sim{x}\cdots|a|x)$ are the third and fourth weights in
(T1). Note that $(U(f^3):K(f^3)) =(U(f^3):K(f^4))=1$ and that the
cosocle of $U(f^3)$ is $L(f^4)$. However, $L(f^4)$ cannot be the
socle of $K(f^3)$. For consider the embedding of $\gl(m+1)\otimes
\gl(1)$-modules $K^{12}(\cdots
a,\widehat{a-1}\sim{x}\cdots|a-1)\otimes L^3(x)\supsetneq
K^{12}(\cdots \widehat{a},a-1\sim{x}\cdots|a)\otimes L^3(x)$, which
we may regard as an embedding of $\mathfrak p$-modules. Inducing to
$\gl(m+1|1)$ we get an embedding $K(f^3)\supsetneq K(f^4)$.  But
$K(f^4)$ is not irreducible, and its socle is not $L(f^4)$. This
implies that $U(f^3)$ cannot have isomorphic socle and cosocle and
hence is not $\tau$-self-dual, contradicting Corollary
\ref{tiltingdual}.

Next consider a weight $f$ of the form in (C2). Since the Verma
modules of the first two weights in (C2) belong to the same
tilting module by Proposition~\ref{condition:can}~(i), we have by
Lemma~\ref{tiltVerma} that $M(f) =U(f)$ or $M(f) =U(f)\oplus
U(f^3)$, where $f^3$ is the third weight in (C2). But the second
possibility cannot occur since $f^3$ is of the form (T1) and
$U(f^3)$ has Verma flag length four by the previous paragraph.

For $f$ of the form in (C4), the second and the third weights are
not comparable under the super Bruhat ordering. Hence using
Proposition~\ref{condition:can} the first three terms lie in the
tilting module $U(f)$. By Lemma~\ref{tiltVerma}, $M(f)$ has to be
indecomposable, and thus equal to $U(f)$.

The same argument for (C4) is applicable to (C7).

Finally, the two cases of (C8) and (C9) in the case when we have
four terms can be verified using Proposition~\ref{condition:can}
(i) and the socle-cosocle argument.
\end{proof}

\begin{remark}
In light of the above theorem, the formulas for canonical basis in
Subsection~\ref{subsec:formula:m11} provides explicit information
on the weights of a Verma flag of any tilting module in $\mathcal
O^+_{m,1|1}$.
\end{remark}

\begin{corollary}
Let $f$ be such that $\Sigma_{f_{13}}\cup\Sigma_{f_{23}} \not
=\emptyset$. Let $X$ and $h$ be as defined in Procedure
\ref{procedure:m11}. Then we have $U(f)=XU(h)$.
\end{corollary}

%\begin{proof}
%Suppose that $f$ is reduced to a typical $g$ via the sequence
%\begin{equation*}
%f\stackrel{X}{\leftarrow}h\stackrel{X_2}{\leftarrow}\cdots
%h_i\stackrel{X_i}{\leftarrow}h_{i+1}\cdots\stackrel{X_r}{\leftarrow}g.
%\end{equation*}
%Suppose that we have $U(f)\not=X U(h)$. This implies that Procedure
%\ref{procedure:m11} would not produce $U(f)$. This is a
%contradiction to Theorem \ref{main:m11}, since we have computed the
%canonical basis elements following Procedure \ref{procedure:m11}.
%\end{proof}

%%
%%
%%
\section{The category of $\gl(2|1)$-modules}
\label{sec:gl21}

In this section we will work out explicitly the Verma flag
structures for tilting modules, projective modules, and the
composition series of Verma modules in the category $\mathcal
O_{2|1}^+$. The results here can be generalized to the category
$\mathcal O_{m+1|1}^+$ in Section~\ref{sec:categorym11} readily
and to the category $\mathcal O_{1+1|n}^+$ in
Section~\ref{canonical:11n} with more complicated notations.

\subsection{The main tools}
Denote by $P(\la)$ the projective cover of $L(\la)$. By abuse of
notations, we shall also write $P(f_\la) =P(\la).$ Recall the BGG
reciprocity for projective modules:
\begin{eqnarray} \label{BGGrec}
(P(f_\la): K(f_\mu)) =[K(f_\mu): L(f_\la)].
\end{eqnarray}
By [Br2, (7.4)],
 %
%$$(U(\la): M(\mu)) =[M(\beta -w_0 \mu): L(\beta -w_0\la)].$$
 %
\begin{eqnarray} \label{tiltrec}
(U(f_\la): K(f_\mu)) =[K(-f_\mu): L(-f_\la)].
\end{eqnarray}

%(\ref{tiltrec}) will be used to work out the composition factors
%of a Verma module, and (\ref{BGGrec}) will be used to work out the
%Verma flag structures of a projective modules.

In the following diagrams, $\bar{i}$ (for $i>0$) denotes $-i$, and
the weights are described using elements in $\Z^{2|1}$ via the
bijection $X_{2|1} \cong \Z^{2|1}, \la \mapsto f_\la$. We will be
only concerned about the block $\B$ of $K(00|0)$ in the category
$\mathcal O^+_{2|1}$. Any block of atypicality $1$ in the category
$\mathcal O^+_{2|1}$ is isomorphic to $\B$. A block of atypicality
$0$ is very simple and will be omitted.

\newpage
\subsection{The poset of weights in the block $\B$}
The poset of $\rho$-shifted weights in $\Z^{2|1}$ for the block
$\B$ is listed in the following diagram. Our convention is that
arrows point to lower weights in the super Bruhat ordering.

$$\CD
 \vdots @>>> \vdots \\
 @AAA @AAA \\
 0\bar{3}|\bar{3}  @>>> \bar{3}0|\bar{3} \\
 @AAA @AAA \\
 0\bar{2}|\bar{2}  @>>> \bar{2}0|\bar{2} \\
 @AAA @AAA \\
 0\bar{1}|\bar{1}  @>>> \bar{1}0|\bar{1} \\
@AAA  \\
 00|0  \\
 @AAA  \\
 01|1 @<<< 10|1 \\
 @AAA @AAA \\
 02|2 @<<< 20|2\\
 @AAA @AAA \\
 03|3 @<<< 30|3\\
 @AAA @AAA \\
 \vdots @<<< \vdots
  \endCD $$

\pagebreak

\subsection{The Verma flag structures of tilting modules in $\B$}
\label{subsec:Vermatilt}

Based on Theorem~\ref{canonical11n} (with $n=1$) or
Theorem~\ref{main:m11} (with $m=1$) and the explicit formulas for
canonical basis in Subsection~\ref{subsec:formula:m11}, we list
the weights of the Verma modules (each with multiplicity 1) which
appear in a Verma flag of a tilting module $U(f)$ in the block
$\B$ as follows. Recall from Theorem~\ref{canonical11n} that every
such $U(f)$ has a simple cosocle.
\begin{eqnarray*}
U(0\bar{i}|\bar{i})
 &\approx & \CD
 0\overline{i+1}|\overline{i+1}  @>>> \overline{i+1}0|\overline{i+1} \\
 @AAA @AAA \\
 0\bar{i}|\bar{i}  @>>> \bar{i}0|\bar{i} \endCD\qquad (i \ge 1)
\\ \\ \\
U(\bar{i}0|\bar{i}) &\approx & \CD
\overline{i+1}0|\overline{i+1} \\
@AAA \\
\bar{i}0|\bar{i} \endCD   \qquad\qquad (i \ge 1)
\\ \\ \\
U(00|0) &\approx &
 \CD
 0\bar{1}|\bar{1}  @>>> \bar{1}0|\bar{1} \\
@AAA  \\
 00|0
\endCD,
\qquad
U(01|1)\approx
 \CD
 00|0  @>>> \bar{1}0|\bar{1} \\
 @AAA  \\
 01|1
\endCD
\\ \\ \\
U(10|1) &\approx &
 \CD
 00|0  \\
 @AAA  \\
 01|1 @<<< 10|1
\endCD
\\ \\ \\
U(0j|j) &\approx & \CD
 0,j-1|j-1 \\
 @AAA \\
 0j|j \\
\endCD  \qquad\qquad (j \ge 2)
\\ \\ \\
U(j0|j) &\approx & \CD
 0,j-1|j-1 @<<< j-1,0|j-1 \\
 @AAA @AAA \\
 0j|j @<<< j0|j\\
\endCD  \qquad (j \ge 2)
\end{eqnarray*}

\pagebreak

\subsection{The composition series of Verma modules in $\B$}
\label{subsec:composition}

The weights of the composition factors of a Verma module $K(f)$ in
the block $\B$ are listed as follows. The calculation is based on
(\ref{tiltrec}) and the Verma flag structure of tilting modules in
Subsection~\ref{subsec:Vermatilt}.
\begin{eqnarray*}
K(0\bar{i}|\bar{i})
 &\approx & \CD
 0\overline{i+1}|\overline{i+1}  @>>> \overline{i+1}0|\overline{i+1} \\
 @AAA @AAA \\
 0\bar{i}|\bar{i}  @>>> \bar{i}0|\bar{i} \endCD \qquad (i \ge 1)
\\ \\ \\
K(\bar{i}0|\bar{i}) &\approx & \CD
\overline{i+1}0|\overline{i+1} \\
@AAA \\
\bar{i}0|\bar{i} \endCD   \qquad\qquad (i \ge 1)
\\ \\ \\
K(00|0) &\approx &
 \CD
 0\bar{1}|\bar{1}  @>>> \bar{1}0|\bar{1} \\
@AAA  \\
 00|0
\endCD,
\qquad
K(01|1) \; \approx \;
 \CD
 00|0  \\
 @AAA  \\
 01|1
\endCD
\\ \\ \\
K(10|1) &\approx &
 \CD
 0\bar{1}|\bar{1}  \\
 @AAA  \\
 00|0  \\
 @AAA  \\
 01|1 @<<< 10|1
\endCD
\\ \\ \\
K(0j|j) &\approx  & \CD
 0,j-1|j-1 \\
 @AAA \\
 0j|j \\
\endCD  \qquad\qquad (j \ge 2)
\\ \\ \\
K(j0|j) &\approx & \CD
 0,j-1|j-1 @<<< j-1,0|j-1 \\
 @AAA @AAA \\
 0j|j @<<< j0|j\\
\endCD\qquad (j \ge 2)
\end{eqnarray*}

\pagebreak
\subsection{The Verma flag structures of projective modules in $\B$}
\label{subsec:proj}

The weights of the Verma modules (each with multiplicity 1) which
appear in a Verma flag of a projective module in the block $\B$
are listed as follows. The calculation is based on (\ref{BGGrec})
and Subsection~\ref{subsec:composition}.
\begin{eqnarray*}
P(\bar{i}0|\bar{i})
 &\approx & \CD
 0\overline{i}|\overline{i}  @>>> \overline{i}0|\overline{i} \\
 @AAA @AAA \\
 0\overline{i-1}|\overline{i-1}  @>>> \overline{i-1}0|\overline{i-1}
 \endCD  \qquad (i \ge 2)
\\ \\ \\
P(0\bar{i}|\bar{i})
 &\approx & \CD
0\overline{i}|\overline{i} \\
@AAA \\
0\overline{i-1}|\overline{i-1}
 \endCD   \qquad\qquad (i \ge 2)
\\ \\ \\
P(\bar{1}0|\bar{1}) &\approx &
 \CD
 0\bar{1}|\bar{1}  @>>> \bar{1}0|\bar{1} \\
@AAA  \\
 00|0
\endCD ,
\qquad
P(0\bar{1}|\bar{1}) \;\approx\;
 \CD
 0\bar{1}|\bar{1} \\
@AAA  \\
 00|0   @<<< 10|1
\endCD
\\ \\ \\
P(00|0) &\approx &
 \CD
 00|0  \\
 @AAA  \\
 01|1 @<<< 10|1
\endCD
\\ \\ \\
P(j0|j) &\approx & \CD
 j0|j \\
 @AAA \\
 j+1,0|j+1 \\
\endCD  \qquad\qquad (j \ge 1)
\\ \\ \\
P(0j|j) &\approx & \CD
 0j|j @<<< j0|j \\
 @AAA @AAA \\
 0,j+1|j+1 @<<< j+1,0|j+1
\endCD
\qquad (j \ge 1)
\end{eqnarray*}
\subsection{The projective tilting modules in $\B$}
\label{subsec:bij}

By Theorem~\ref{canonical11n} for $n=1$, the tilting module
$U(0\overline{i-1}|\overline{i-1})$ in the block $\B$ has a simple
cosocle $L(\overline{i}0|\overline{i})$ for $i\ge 1$. Thus the
nontrivial $\gl(2|1)$-module homomorphism $\pi_{-i}:
P(\overline{i}0|\overline{i}) \longrightarrow
U(0\overline{i-1}|\overline{i-1})$ has to be surjective. By
observation from the previous diagrams,
$U(0\overline{i-1}|\overline{i-1})$ and
$P(\overline{i}0|\overline{i})$ have the same Verma flag
multiplicity and thus the same composition series. It follows that
$\pi_{-i}$ is indeed an isomorphism.

Similarly, there is a $\gl(2|1)$-module isomorphism $\pi_i:
P(0i|i) \longrightarrow U(i+1,0|i+1)$ for $i \ge 0$. Again by
observation from the diagrams, the remaining tilting modules are
not projective.

The above discussion can be summarized in the following.

\begin{proposition}
The projective tilting modules in the category $\mathcal O^+_{2|1}$
consist of $U(i0|i)$ for $i\ge 0$ and $U(0j|j)$ for $j<0$.
\end{proposition}

\end{document}